\documentclass[12pt,a4paper]{article}

\usepackage{amsfonts,amssymb}
\usepackage{amsmath,amsthm,MnSymbol}
\usepackage[all]{xy}

\newtheorem{proposition}{Proposition}[section]
\newtheorem{lemma}[proposition]{Lemma}

\newtheorem{theorem}[proposition]{Theorem}

\theoremstyle{definition}

\theoremstyle{remark}

\newcommand{\thlabel}[1]{\label{th:#1}}
\newcommand{\thref}[1]{Theorem~\ref{th:#1}}
\newcommand{\selabel}[1]{\label{se:#1}}
\newcommand{\seref}[1]{Section~\ref{se:#1}}
\newcommand{\lelabel}[1]{\label{le:#1}}
\newcommand{\leref}[1]{Lemma~\ref{le:#1}}
\newcommand{\prlabel}[1]{\label{pr:#1}}
\newcommand{\prref}[1]{Proposition~\ref{pr:#1}}

\newcommand{\eqlabel}[1]{\label{eq:#1}}
\newcommand{\equref}[1]{(\ref{eq:#1})}

\title{\textbf{On a class of Fock-like representations for Lie superalgebras}}

\author{\textbf{\textsc{K. Kanakoglou}} \\ {\small kanakoglou@hotmail.com} {\scriptsize and}
{\small kanakoglou@ifm.umich.mx} \\
\small{School of Mathematics},
\small{Aristotle University of Thessaloniki - \textsc{Auth}} \\
\small{54124 Thessaloniki, \textsc{Greece}} \\  \\
\textbf{\textsc{A. Herrera-Aguilar}} \\ {\small alfredo.herrera.aguilar@gmail.com} \\
\small{Instituto de F\'{\i}sica y Matem\'{a}ticas - \textsc{Ifm}}, \\
\small{Universidad Michoacana de San Nicol\'{a}s de Hidalgo - \textsc{Umsnh}} \\
\small{Edificio C-3, Cd. Universitaria, CP 58040},
\small{Morelia, Michoac\'{a}n, \textsc{Mexico}}}

\date{}

\begin{document}

\maketitle

\begin{abstract}
{\small
Utilizing Lie superalgebra (LS) realizations via the Relative Parabose Set algebra $P_{BF}$, combined with earlier results on the Fock-like representations of $P_{BF}^{(1,1)}$, we proceed to the construction of a family of Fock-like representations of LSs: these are infinite dimensional, decomposable super-representations, which are parameterized by the value of a positive integer $p$. They can be constructed for any LS $L$, either initiating from a given $2$-dimensional, $\mathbb{Z}_{2}$-graded representation of $L$ or using its inclusion as a subalgebra of $P_{BF}^{(1,1)}$. As an application we proceed in studying decompositions with respect to various low-dimensional Lie algebras and superalgebras.
}
\end{abstract}


\textbf{PACS2010:} 03.65.Fd, 02.20.Sv, 02.10.Hh, 03.65.-w, 11.30.Pb, 12.60.Jv

\textbf{keywords:} Lie superalgebras, paraparticles, realizations, Fock spaces, representations, decomposability, branching rules, multiplicities

\section{Introduction}

In \cite{KaDaHa, KaDaHa2} we develop a family of realizations of an arbitrary Lie superalgebra (LS) via an algebra combining both parabosonic and parafermionic degrees of freedom, known in the literature as the Relative Parabose Set \cite{GreeMe} $P_{BF}$. It is an algebra with infinite generators (i.e. infinite parabosonic and parafermionic degrees of freedom) in the general case. In \cite{KaDaHa3, KaDaHa4, Ya2} the authors study the Relative Parabose Set algebra but for the special case for which we have a single parabosonic and a single parafermionic degree of freedom. We will use the notation $P_{BF}^{(1,1)}$ for this algebra. It is an algebra with 4 generators (two parabosonic and two parafermionic generators corresponding to a single degree of freedom each).

The purpose of this paper is to combine the results of \cite{KaDaHa, KaDaHa2} together with the results of \cite{KaDaHa3, KaDaHa4, Ya2} in order to construct a family (parametrized by the values of a positive integer $p$) of $\mathbb{Z}_{2}$-graded, infinite dimensional, decomposable representations for an arbitrary LS. The structure of the article is as follows:

In \seref{realizations}, we review results \cite{KaDaHa, KaDaHa2, Ya1} about the structure and the properties of the Relative Parabose Set $P_{BF}$ in the general case of infinite degrees of freedom (infinite generators as an algebra). Moreover, we investigate some of its subalgebras, we introduce realizations of an arbitrary LS via elements of $P_{BF}$ and we study the algebraic properties of the constructed realizations.

In \seref{singleFockspace}, we review from \cite{KaDaHa3, KaDaHa4, Ya2} the construction of the Fock-like representations for the Relative Parabose Set $P_{BF}^{(1,1)}$ algebra. We introduce some notation and terminology to be used in the sequel and we give explicitly the structure of the carrier space, the formulae for the action of the generators, the irreducibility and the $(\mathbb{Z}_{2}\times \mathbb{Z}_{2})$-grading.

In \seref{Focklikemodule}, we combine the above results and proceed to the construction of a family of infinite dimensional, $\mathbb{Z}_{2}$-graded, decomposable representations for a Lie superalgebra $L$. This family of representations is parameterized by the values of a positive integer $p$. We will call these representations \textbf{Fock-like representations of $L$}. The only assumption we make is that we have available a 2-dimensional, $\mathbb{Z}_{2}$-graded representation of the Lie superalgebra $L$.

Finally in \seref{applLSrepr}, we proceed in some applications: Using realizations of various low-dimensional Lie algebras and superalgebras via paraparticles (either with the mappings described in \seref{realizations} or using their inclusions as subalgebras of $P_{BF}^{(1,1)}$) and the Fock-like representations of $P_{BF}^{(1,1)}$ we proceed in studying decompositions of the Fock-like spaces under various different actions.

Before closing this introduction we remark that all vector spaces, algebras and tensor products in this article will be considered over the field of complex numbers $\mathbb{C}$ and that the prefix ``super'' will always amount to $\mathbb{Z}_{2}$-graded and used according to preference without further mentioning it.

\section{Lie superalgebra realizations via $P_{BF}$}  \selabel{realizations}

\paragraph{$\bullet$ \textbf{On the multiplicative structure of the Relative Parabose Set $P_{BF}$:}}
The Relative Parabose Set has been historically the only -together with the \emph{Relative Parafermi Set} $P_{FB}$- attempt for a mixture of interacting parabosonic and parafermionic degrees of freedom. We present it here in terms of generators and relations, adopting a handy notation: $P_{BF}$ is generated, as an associative algebra, by the (infinite) generators $b_{i}^{\xi}$, $f_{j}^{\eta}$, for all values $i,j = 1, 2, ...$ and $\xi, \eta = \pm$. The relations satisfied by the above generators are:

The usual trilinear relations of the parabosonic and the parafermionic algebras which can be compactly summarized as
\small{\begin{equation} \eqlabel{parab-paraf}
\begin{array}{c}
\big[ \{ b_{i}^{\xi},  b_{j}^{\eta}\}, b_{k}^{\epsilon}  \big] = (\epsilon - \eta)\delta_{jk}b_{i}^{\xi} + (\epsilon -
 \xi)\delta_{ik}b_{j}^{\eta}    \\
\big[ [ f_{i}^{\xi},  f_{j}^{\eta} ], f_{k}^{\epsilon}  \big] = \frac{1}{2}(\epsilon - \eta)^{2} \delta_{jk}f_{i}^{\xi} - \frac{1}{2}(\epsilon -
 \xi)^{2} \delta_{ik}f_{j}^{\eta}   \\
 \end{array}
\end{equation}}
for all values $i, j, k = 1, 2, ..., $ and $\xi, \eta, \epsilon = \pm$, \ together with the mixed trilinear relations
{\small
\begin{equation} \eqlabel{parabparaf}
\begin{array}{c}
\big[ \{ b_{k}^{\xi},  b_{l}^{\eta}\}, f_{m}^{\epsilon}  \big] = \big[ [ f_{k}^{\xi},  f_{l}^{\eta} ], b_{m}^{\epsilon}  \big] = 0    \\
\big[ \{ f_{k}^{\xi},  b_{l}^{\eta}\}, b_{m}^{\epsilon}  \big] = (\epsilon - \eta) \delta_{lm} f_{k}^{\xi},  \ \ \ \
\big\{ \{ b_{k}^{\xi},  f_{l}^{\eta}\}, f_{m}^{\epsilon}  \big\} = \frac{1}{2}(\epsilon - \eta)^{2} \delta_{lm} b_{k}^{\xi}
\end{array}
\end{equation}
}
for all values $k, l, m = 1, 2, ..., $ and $\xi, \eta, \epsilon = \pm$, \ which represent a kind of algebraically established interaction between parabosonic and parafermionic degrees of freedom and characterize the relative parabose set. If in relations \equref{parabparaf} we write down all combinations of values for the indices $\xi, \eta, \epsilon$ then 28 (algebraically) independent relations emerge. These can be found written explicitly in \cite{KaDaHa}.

One can easily observe that \equref{parab-paraf} involve only the parabosonic and the parafermionic degrees of freedom separately while the ``interaction'' relations \equref{parabparaf} mix the parabosonic with the parafermionic degrees of freedom according to the recipe proposed in \cite{GreeMe}. In all the above and in what follows, we use the notation $[x, y]\equiv xy-yx$ (i.e.: the ``commutator'') and the notation $\{x, y \}\equiv xy+yx$ (i.e.: the ``anticommutator''), for any $x,y \in P_{BF}$.

If we consider finite number of generators, for example the $2m$ parabosonic generators $b_{i}^{+}, b_{i}^{-}$ ($i = 1, 2, ..., m$) and the $2n$ parafermionic generators $f_{j}^{+}, f_{j}^{-}$ ($j = 1, 2, ..., n$) we will use the notation $P_{BF}^{(m,n)}$.

We will now review results presented and proved in \cite{KaDaHa, Ya1} regarding the braided, graded algebraic structure of the Relative Parabose Set $P_{BF}$ and the construction of Lie superalgebra realizations via the use of $P_{BF}$.  For the results presented in the rest of this section concerning the graded structure one can also see \cite{Ya1}, while for detailed computations and proofs of the results concerning the braided group structure, subalgebras and finally the paraparticle realizations one can see \cite{KaDaHa}.

\paragraph{$\bullet$ \textbf{On the $\mathbf{(\mathbb{Z}_{2} \times \mathbb{Z}_{2})}$-Graded structure of $P_{BF}$:}} The following proposition describes the $(\mathbb{Z}_{2} \times \mathbb{Z}_{2})$-graded structure of the Relative Parabose Set algebra $P_{BF}$ (For the proof see \cite{Ya1}):
\begin{proposition} \prlabel{gradedstruct}
The relative parabose set $P_{BF}$ is the universal enveloping algebra UEA of a
$\theta$-colored ($\mathbb{Z}_{2} \times \mathbb{Z}_{2}$)-graded Lie algebra $L_{\mathbb{Z}_{2} \times \mathbb{Z}_{2}}$.
This implies that $P_{BF}$ is a ($\mathbb{Z}_{2} \times \mathbb{Z}_{2}$)-graded associative algebra
{\small
\begin{equation}
P_{BF} \cong \mathbb{U}(L_{\mathbb{Z}_{2} \times \mathbb{Z}_{2}})
\end{equation}
}
Its generators are homogeneous elements in the above gradation, with the paraboson generators $b_{k}^{+}$, $b_{l}^{-}$,
$k,l = 1, 2, ...$ spanning the $L_{10}$ subspace of $L_{\mathbb{Z}_{2} \times \mathbb{Z}_{2}}$, and the parafermion generators
$f_{\alpha}^{+}$, $f_{\beta}^{-}$, $\alpha, \beta = 1, 2, ...$ spanning the $L_{11}$ subspace of $L_{\mathbb{Z}_{2} \times \mathbb{Z}_{2}}$, thus
their grades are given as follows
{\small
\begin{equation} \eqlabel{gradPBF1}
\begin{array}{ccc}
deg(b_{k}^{\varepsilon}) = (1,0) & , & deg(f_{\alpha}^{\eta}) = (1,1)
\end{array}
\end{equation}
}
where $\varepsilon, \eta = \pm$. At the same time the polynomials $\{ b_{k}^{\epsilon}, b_{l}^{\eta} \}$ and $[ f_{\alpha}^{\epsilon}, f_{\beta}^{\eta} ]$ $\forall$ $k, l, \alpha, \beta = 1,2,... \ $ and $\forall$ $\epsilon, \eta = \pm$ span the subspace $L_{00}$ of $L_{\mathbb{Z}_{2} \times \mathbb{Z}_{2}}$,
and the polynomials $\{ f_{\alpha}^{\epsilon}, b_{k}^{\eta} \}$ $\forall$ $k, \alpha = 1,2,... \ $ and $\forall$ $\epsilon, \eta = \pm$ span the subspace $L_{01}$ of $L_{\mathbb{Z}_{2} \times \mathbb{Z}_{2}}$. Consequently their grades are given as follows
{\small
\begin{equation}  \eqlabel{gradPBF2}
\begin{array}{ccc}
deg(\{ b_{k}^{\epsilon}, b_{l}^{\eta} \}) = deg([ f_{\alpha}^{\epsilon}, f_{\beta}^{\eta} ]) = (0,0) & , &
deg(\{ f_{\alpha}^{\epsilon}, b_{k}^{\eta} \}) = (0,1)
\end{array}
\end{equation}
}
Finally the color function used in the above construction is given
{\small
\begin{equation} \eqlabel{colfunctRelParSe}
\begin{array}{c}
\theta:\big( \mathbb{Z}_{2} \times \mathbb{Z}_{2} \big) \times \big( \mathbb{Z}_{2} \times \mathbb{Z}_{2} \big) \rightarrow \mathbb{C}^{*} \\
\theta(a,b) = (-1)^{(a_{1}b_{1} + a_{2}b_{2})}
\end{array}
\end{equation}
}
for all $a = (a_{1}, a_{2}), b = (b_{1}, b_{2}) \in \mathbb{Z}_{2} \times \mathbb{Z}_{2}$, and the operations in the exponent are considered in the
$\mathbb{Z}_{2}$ ring.
\end{proposition}
Notice that in the above -and in what follows- we are employing the additive notation for the Klein group $\mathbb{Z}_{2} \times \mathbb{Z}_{2}$:
{\scriptsize
$$
\begin{array}{ccc}
\begin{array}{c}
\mathbf{\mathbb{Z}_{2} \times \mathbb{Z}_{2} \cong \mathbb{Z}_{2} \oplus \mathbb{Z}_{2}} \ \textbf{group}:   \\  \\
\begin{array}{|c||c|c|c|c|}
  \hline
       \textbf{+}         & \mathbf{(0,0)}  & \mathbf{(0,1)} & \mathbf{(1,0)}  &  \mathbf{(1,1)} \\ \hline \cline{1-5}
 \mathbf{(0,0)} & (0,0) & (0,1) & (1,0)   &   (1,1)            \\   \hline
 \mathbf{(0,1)} & (0,1) & (0,0) & (1,1)   &  (1,0) \\   \hline
 \mathbf{(1,0)} & (1,0) & (1,1) & (0,0)   &   (0,1)     \\   \hline
 \mathbf{(1,1)} & (1,1) & (1,0) & (0,1)   &   (0,0)      \\
\hline
\end{array}
\end{array}     &  \leftrightsquigarrow   &  
                                             \begin{array}{c}
                                             \{(0,0), (0,1)\} \ is \ a \ subgroup \\
                                             of \ the \ \mathbb{Z}_{2} \times \mathbb{Z}_{2} \ group \\
                                             isomorphic \ to \ the \\
                                             \mathbf{\mathbb{Z}_{2} = \{0, 1\} \ \textbf{group}}:  \\   \\
                                             \begin{array}{|c||c|c|}
                                               \hline
                                               \textbf{+} & \mathbf{0} & \mathbf{1} \\   \hline \cline{1-3}
                                               \mathbf{0} & 0 & 1 \\   \hline
                                               \mathbf{1} & 1 & 0 \\
                                               \hline
                                             \end{array}
                                             \end{array}
\end{array}
$$
}

\paragraph{$\bullet$ \textbf{On some subalgebras of $P_{BF}$:}} The following proposition singles out a particularly useful subalgebra of $P_{BF}$ which will be utilized in the sequel:
\begin{proposition}     \prlabel{subalgPbf}
The linear subspace of the Relative Parabose set $P_{BF}$ (or of the ($\mathbb{Z}_{2} \times \mathbb{Z}_{2}, \theta$)-Lie algebra $L_{\mathbb{Z}_{2} \times \mathbb{Z}_{2}}$) spanned by the elements of the form $\{ b_{k}^{\epsilon}, b_{l}^{\eta} \}$, $[ f_{\alpha}^{\epsilon}, f_{\beta}^{\eta} ]$ and $\{ f_{\alpha}^{\epsilon}, b_{k}^{\eta} \}$ for all $k, l, \alpha, \beta = 1,2,... \ $ and for all $\epsilon, \eta = \pm$ is a $\mathbb{Z}_{2}$-graded Lie algebra (or equivalently a Lie superalgebra). The UEA of this Lie superalgebra is a subalgebra of $P_{BF}$.
\end{proposition}
\begin{proof}
For the proof see \cite{KaDaHa}.
\end{proof}
Let us see this last statement in a little more detail: Employing the commutation factor $\theta$ given in \equref{colfunctRelParSe} and the gradation described in \prref{gradedstruct} we get the following relations inside the relative parabose set
\begin{equation} \eqlabel{Z2subZ22}
\begin{array}{c}
L_{00} \ni \langle L_{00}, L_{00} \rangle =  [ L_{00}, L_{00} ]  \rightsquigarrow  \textrm{commutator} \\
    \\
L_{01} \ni \langle L_{00}, L_{01} \rangle =  [ L_{00}, L_{01} ]  \rightsquigarrow  \textrm{commutator}       \\
    \\
L_{00} \ni \langle L_{01}, L_{01} \rangle =  \{ L_{01}, L_{01} \}   \rightsquigarrow \textrm{anticommutator}      \\
\end{array}
\end{equation}
where $\langle .., .. \rangle : L \times L \rightarrow L$ is the non-associative multiplication of the ($\mathbb{Z}_{2} \times \mathbb{Z}_{2}, \theta$)-Lie algebra $L_{\mathbb{Z}_{2} \times \mathbb{Z}_{2}}$ described in \prref{gradedstruct}

Now we will proceed in computing the relations in this Lie superalgebra:
\begin{lemma}  \lelabel{commrel}
The non-associative multiplication $\langle .., .. \rangle: \mathbb{L} \times \mathbb{L} \rightarrow \mathbb{L}$ of the Lie superalgebra $\mathbb{L} = L_{00} \oplus L_{01}$ described in \prref{subalgPbf} is determined by the values of the commutators $[ .., .. ]$ and the anticommutators $\{ .., .. \}$ in $P_{BF}$. Consequently we have the following relations (inside $\mathbb{U}(L_{00} \oplus L_{01})$):
\begin{equation} \eqlabel{liesuperalgmultipl1}
\begin{array}{c}
\mathbf{\big\langle \{ b_{i}^{\xi}, b_{j}^{\eta} \}, \{ b_{k}^{\epsilon}, b_{l}^{\phi} \} \big\rangle} \equiv  \\
\equiv \mathbf{\big[ \{ b_{i}^{\xi}, b_{j}^{\eta} \}, \{ b_{k}^{\epsilon}, b_{l}^{\phi} \} \big]} = (\epsilon - \eta) \delta_{jk} \{ b_{i}^{\xi}, b_{l}^{\phi} \} + (\epsilon - \xi) \delta_{ik} \{ b_{j}^{\eta}, b_{l}^{\phi} \} + \\
+ (\phi - \eta) \delta_{jl} \{ b_{i}^{\xi}, b_{k}^{\epsilon} \} + (\phi - \xi) \delta_{il} \{ b_{j}^{\eta}, b_{k}^{\epsilon} \}
\end{array}
\end{equation}
\begin{equation} \eqlabel{liesuperalgmultipl2}
\begin{array}{c}
\mathbf{\big\langle \{ b_{i}^{\xi}, b_{j}^{\eta} \}, [ f_{k}^{\epsilon}, f_{l}^{\phi} ] \big\rangle} \equiv \mathbf{\big[ \{ b_{i}^{\xi}, b_{j}^{\eta} \}, [ f_{k}^{\epsilon}, f_{l}^{\phi} ] \big]} = 0
\end{array}
\end{equation}
\begin{equation} \eqlabel{liesuperalgmultipl3}
\begin{array}{c}
\mathbf{\big\langle \{ b_{i}^{\xi}, b_{j}^{\eta} \}, \{ f_{k}^{\epsilon}, b_{l}^{\phi} \} \big\rangle} \equiv \\
\equiv \mathbf{\big[ \{ b_{i}^{\xi}, b_{j}^{\eta} \}, \{ f_{k}^{\epsilon}, b_{l}^{\phi} \} \big]} = (\phi - \eta) \delta_{jl} \{ f_{k}^{\epsilon}, b_{i}^{\xi} \} + (\phi - \xi) \delta_{il} \{ f_{k}^{\epsilon}, b_{j}^{\eta} \}
\end{array}
\end{equation}
\begin{equation} \eqlabel{liesuperalgmultipl4}
\begin{array}{c}
\mathbf{\big\langle [ f_{i}^{\xi}, f_{j}^{\eta} ], [ f_{k}^{\epsilon}, f_{l}^{\phi} ] \big\rangle} \equiv \\
\equiv \mathbf{\big[ [ f_{i}^{\xi}, f_{j}^{\eta} ], [ f_{k}^{\epsilon}, f_{l}^{\phi} ] \big]} = \frac{1}{2}(\phi - \eta)^{2}\delta_{jl} [ f_{k}^{\epsilon}, f_{i}^{\xi} ] + \frac{1}{2}(\phi - \xi)^{2}\delta_{il} [ f_{j}^{\eta}, f_{k}^{\epsilon} ] +  \\
+ \frac{1}{2}(\epsilon - \eta)^{2}\delta_{jk} [ f_{i}^{\xi}, f_{l}^{\phi} ] + \frac{1}{2}(\epsilon - \xi)^{2}\delta_{ik} [ f_{l}^{\phi}, f_{j}^{\eta} ]
\end{array}
\end{equation}
\begin{equation} \eqlabel{liesuperalgmultipl5}
\begin{array}{c}
\mathbf{\big\langle [ f_{i}^{\xi}, f_{j}^{\eta} ], \{ f_{k}^{\epsilon}, b_{l}^{\phi} \} \big\rangle} \equiv \\
\equiv \mathbf{\big[ [ f_{i}^{\xi}, f_{j}^{\eta} ], \{ f_{k}^{\epsilon}, b_{l}^{\phi} \} \big]} = \frac{1}{2} (\epsilon - \eta)^{2} \delta_{jk} \{ f_{i}^{\xi}, b_{l}^{\phi} \} - \frac{1}{2} (\epsilon - \xi)^{2} \delta_{ik} \{ f_{j}^{\eta}, b_{l}^{\phi} \}
\end{array}
\end{equation}
\begin{equation} \eqlabel{liesuperalgmultipl6}
\begin{array}{c}
\mathbf{\big\langle \{ f_{k}^{\xi}, b_{l}^{\eta} \}, \{ f_{i}^{\epsilon}, b_{j}^{\phi} \} \big\rangle} \equiv  \\
\equiv \mathbf{\big\{ \{ f_{k}^{\xi}, b_{l}^{\eta} \}, \{ f_{i}^{\epsilon}, b_{j}^{\phi} \} \big\}} = (\phi - \eta) \delta_{jl} [ f_{k}^{\xi}, f_{i}^{\epsilon} ] + \frac{1}{2}(\epsilon - \xi)^{2} \delta_{ik} \big\{ b_{l}^{\eta}, b_{j}^{\phi} \big\}
\end{array}
\end{equation}
for all values $i,j,k,l = 1, 2, ... $ and $\xi, \eta, \epsilon, \phi = \pm$.
\end{lemma}
It is interesting to note the following thing at this point: In the case of finite degrees of freedom ($m$ parabosonic degrees of freedom and $n$ parafermionic degrees of freedom, i.e. the parabosonic and parafermionic indices ranging in $1, 2, ..., m$ and $1, 2, ..., n$ respectively) the commutation relations \equref{liesuperalgmultipl1} are exactly the commutation relations of the Lie algebra $\mathfrak{sp}(2m)$ which is exactly the even part of the Lie superalgebra $\mathfrak{osp}(1/2m) \equiv B(0,m)$ \cite{Pal3}. On the other hand, the commutation relations \equref{liesuperalgmultipl4} are exactly the commutation relations of the Lie algebra $\mathfrak{so}(2n)$ \cite{Kata,RySu}. Finally taking into account \equref{liesuperalgmultipl2} as well, we arrive at the conclusion that the even part $L_{00}$ of the Lie superalgebra $\mathbb{L} = L_{00} \oplus L_{01}$ described in \prref{subalgPbf} and \leref{commrel} is isomorphic (as a Lie algebra) to $\mathfrak{sp}(2m) \oplus \mathfrak{so}(2n)$.

Finally, the following lemma singles out one more useful subalgebra of the Relative Parabose Set algebra $P_{BF}$:
\begin{lemma} \lelabel{generalinearsuperalg}
The Lie superalgebra $\mathfrak{gl}(m/n)$ is a subalgebra of the LS described in \prref{subalgPbf} and \leref{commrel}
\end{lemma}
\begin{proof}
We just have to consider the subspace of the LS $\mathbb{L} = L_{00} \oplus L_{01}$ (previously described), which is spanned by the elements $\{ b_{i}^{+}, b_{j}^{-} \}$, $[ f_{k}^{+}, f_{l}^{-} ]$, $\{ b_{i}^{+}, f_{k}^{-} \}$, $\{ f_{l}^{+}, b_{j}^{-} \}$ with the parabosonic indices taking the values $1, 2, ..., m$ and the parafermionic indices taking the values $1, 2, ..., n$. The rest is a straightforward consequence of equations \equref{liesuperalgmultipl1} - \equref{liesuperalgmultipl6} of \leref{commrel}.

For clarity we reproduce here the relations
\begin{equation} \eqlabel{genlinsuperalg1}
\big[ \{ b_{i}^{+}, b_{j}^{-} \}, \{ b_{k}^{+}, b_{l}^{-} \} \big] = 2 \delta_{jk} \{ b_{i}^{+}, b_{l}^{-} \} - 2 \delta_{il} \{ b_{k}^{+}, b_{j}^{-} \}
\end{equation}
\begin{equation}  \eqlabel{genlinsuperalg2}
\big[ \{ b_{i}^{+}, b_{j}^{-} \}, [ f_{k}^{+}, f_{l}^{-} ] \big] = 0
\end{equation}
\begin{equation}   \eqlabel{genlinsuperalg3}
\big[ [ f_{i}^{+}, f_{j}^{-} ], [ f_{k}^{+}, f_{l}^{-} ] \big] = 2 \delta_{jk} [ f_{i}^{+}, f_{l}^{-} ] - 2 \delta_{il} [ f_{k}^{+}, f_{j}^{-} ]
\end{equation}
\begin{equation}  \eqlabel{genlinsuperalg4}
\big\{ \{ f_{k}^{+}, b_{l}^{-} \}, \{ f_{i}^{+}, b_{j}^{-} \} \big\} = \big\{ \{ f_{k}^{-}, b_{l}^{+} \}, \{ f_{i}^{-}, b_{j}^{+} \} \big\} = 0
\end{equation}
\begin{equation}  \eqlabel{genlinsuperalg5}
\big\{ \{ f_{k}^{+}, b_{l}^{-} \}, \{ f_{i}^{-}, b_{j}^{+} \} \big\} = 2 \delta_{jl} [f_{k}^{+}, f_{i}^{-}] + 2 \delta_{ik} \{ b_{j}^{+}, b_{l}^{-} \}
\end{equation}
\begin{equation}  \eqlabel{genlinsuperalg6}
\big[ \{ b_{i}^{+}, b_{j}^{-} \}, \{ f_{k}^{+}, b_{l}^{-} \} \big] = -2 \delta_{il} \{ f_{k}^{+}, b_{j}^{-} \}
\end{equation}
\begin{equation} \eqlabel{genlinsuperalg7}
\big[ \{ b_{i}^{+}, b_{j}^{-} \}, \{ f_{k}^{-}, b_{l}^{+} \} \big] = 2 \delta_{jl} \{ f_{k}^{-}, b_{i}^{+} \}
\end{equation}
\begin{equation}  \eqlabel{genlinsuperalg8}
\big[ [ f_{i}^{+}, f_{j}^{-} ], \{ f_{k}^{+}, b_{l}^{-} \} \big] = 2 \delta_{jk} \{ f_{i}^{+}, b_{l}^{-} \}
\end{equation}
\begin{equation}  \eqlabel{genlinsuperalg9}
\big[ [ f_{i}^{+}, f_{j}^{-} ], \{ f_{k}^{-}, b_{l}^{+} \} \big] = -2 \delta_{ik} \{ f_{j}^{-}, b_{l}^{+} \}
\end{equation}
Relations \equref{genlinsuperalg1} are the commutation relations of the Lie algebra $\mathfrak{gl}(m)$ while \equref{genlinsuperalg3} are those of $\mathfrak{gl}(n)$. Relations \equref{genlinsuperalg1}, \equref{genlinsuperalg2}, \equref{genlinsuperalg3} determine $\mathfrak{gl}(m) \oplus \mathfrak{gl}(n)$ which is the even part of the Lie superalgebra $\mathfrak{gl}(m/n)$, thus a subalgebra of $L_{00} \cong \mathfrak{sp}(2m) \oplus \mathfrak{so}(2n)$. Finally all the relations \equref{genlinsuperalg1} - \equref{genlinsuperalg9} are exactly the commutation-anticommutation relations of the $(m+n)^{2}$-dimensional Lie superalgebra $\mathfrak{gl}(m/n)$ (see also \cite{Scheu3}).
\end{proof}

\paragraph{$\bullet$ \textbf{On Paraparticle Realizations of Lie superalgebras via $P_{BF}$:}}
Let $L=L_{0} \oplus L_{1}$ be any complex Lie superalgebra of either finite or infinite dimension and let $V = V_{0} \oplus V_{1}$ be a finite
dimensional, complex, super-vector space i.e. $dim_{\mathbb{C}}V_{0} = m$ and $dim_{\mathbb{C}}V_{1} = n$. If $V$ is the carrier space for a
super-representation (or: a $\mathbb{Z}_{2}$-graded representation) of $L$, this is equivalent to the existence of an homomorphism
$P: \mathbb{U}(L) \rightarrow \mathcal{E}nd_{gr}(V)$ of associative superalgebras, from $\mathbb{U}(L)$ to the algebra $\mathcal{E}nd_{gr}(V)$ of $\mathbb{Z}_{2}$-graded
linear maps on $V$.
For any homogeneous element $z \in L$ the image $P(z)$ will be a  $(m+n)\times (m+n)$ matrix of the form
\begin{equation}
\begin{array}{ccc}
 P(z)= \left(\begin{array}{c|c}
              A(z) & B(z) \\ \hline
              C(z) & D(z)
             \end{array}\right)    &  \begin{array}{c}
                                      \nearrow  \\ \\
                                      \searrow
                                      \end{array}   & \begin{array}{c}
                                                      P(X)=
                                                            \left(\begin{array}{ccc}
                                                            A(X) & 0 \\
                                                            0 & D(X)
                                                            \end{array}\right)
                                                            \ \ (if \ z = X \rightsquigarrow \underline{even})
                                                            \\  \\  \\
                                                      P(Y) =
                                                            \left(\begin{array}{ccc}
                                                            0 & B(Y) \\
                                                            C(Y) & 0
                                                            \end{array}\right)
                                                           \ \ (if \ z = Y \rightsquigarrow \underline{odd})
                                                      \end{array}
\end{array}
\end{equation}
where the complex submatrices $A_{m\times m}$, $B_{m\times n}$, $C_{n\times m}$, $D_{n\times n}$,  of $P_{(m+n)\times (m+n)}$ constitute the partitioning ($\mathbb{Z}_{2}$-grading) of the representation.
\begin{proposition}     \prlabel{parapartrealiz}
$\mathbf{(a).}$ The linear map $J_{P_{BF}}: L \rightarrow P_{BF}$ defined by
\begin{equation} \eqlabel{JordgenPBFeqeven}
X_{i} \mapsto J_{P_{BF}}(X_{i}) = \frac{1}{2} \sum_{k,l=1}^{m}A_{kl}(X_{i}) \{ b_{k}^{+}, b_{l}^{-} \} + \frac{1}{2} \sum_{\alpha,\beta=1}^{n}D_{\alpha\beta}(X_{i}) [ f_{\alpha}^{+}, f_{\beta}^{-} ] \\
\end{equation}
for any even element ($Z=X_{i}$) of an homogeneous basis of $L$ and by
\begin{equation} \eqlabel{JordgenPBFeqodd}
Y_{j} \mapsto J_{P_{BF}}(Y_{j}) = \frac{1}{2} \sum_{k=1}^{m}\sum_{\alpha=1}^{n} \Big( B_{k,\alpha}(Y_{j}) \{ b_{k}^{+}, f_{\alpha}^{-} \}
+ C_{\alpha,k}(Y_{j}) \{ f_{\alpha}^{+}, b_{k}^{-} \} \Big) \\
\end{equation}
for any odd element ($Z=Y_{j}$) of an homogeneous basis of $L$, can be extended to an homomorphism of associative algebras
\begin{equation}
J_{P_{BF}}: \mathbb{U}(L) \rightarrow \mathbb{U} \big( \mathfrak{gl}(m/n) \big) \subset \mathbb{U}(\mathbb{L}) \equiv \mathbb{U}(L_{00} \oplus L_{01}) \subset P_{BF}
\end{equation}
between the universal enveloping algebra $\mathbb{U}(L)$ of the Lie superalgebra $L$ and the subalgebra $\mathbb{U} \big( \mathfrak{gl}(m/n) \big)$ of the  relative parabose set $P_{BF}$, in other words it constitutes a \textbf{realization of $L$ with paraparticles}.

$\mathbf{(b).}$ Furthermore, the above constructed homomorphism of associative algebras $J_{P_{BF}}: \mathbb{U}(L) \rightarrow P_{BF}$, is an homomorphism of super-Hopf algebras (equivalently an homomorphism of $\mathbb{Z}_{2}$-graded Hopf algebras) between $\mathbb{U}(L)$ and the $\mathbb{U} \big( \mathfrak{gl}(m/n) \big)$ $\mathbb{Z}_{2}$-graded subalgebra of $P_{BF}$.
\end{proposition}
\begin{proof}
$\underline{(a).}$ The first statement of the proposition is fully justified by the
fact that the linear map $J_{P_{BF}}$ preserves (see \cite{KaDaHa}) for all values $i, j, p, q = 1,2,...$ the Lie superalgebra brackets:
\begin{equation}
\begin{array}{c}
J_{P_{BF}}(\big[ X_{i}, X_{j} \big]) = \big[ J_{P_{BF}}(X_{i}), J_{P_{BF}}(X_{j}) \big]   \\ \\
J_{P_{BF}}(\big[ X_{i}, Y_{p} \big]) = \big[ J_{P_{BF}}(X_{i}), J_{P_{BF}}(Y_{p}) \big]  \\ \\
J_{P_{BF}}(\{ Y_{p}, Y_{q} \}) = \{ J_{P_{BF}}(Y_{p}), J_{P_{BF}}(Y_{q}) \}
\end{array}
\end{equation}
$\underline{(b).}$ To prove the second statement it is enough to apply \prref{subalgPbf} and then compute (see \cite{KaDaHa}) the rhs and the lhs of the following
\begin{equation}
\begin{array}{ccccccccc}
\underline{\Delta} \circ J_{P_{BF}} = (J_{P_{BF}} \otimes J_{P_{BF}}) \circ \Delta_{L},  & & & & \underline{\varepsilon} \circ J_{P_{BF}} = \varepsilon_{L},   &  & & &
\underline{S} \circ J_{P_{BF}} = J_{P_{BF}} \circ S_{L}
\end{array}
\end{equation}
where $\Delta_{L}, \varepsilon_{L}, S_{L}$ are the Lie superalgebra super Hopf structure maps and $\underline{\Delta}, \underline{\varepsilon}, \underline{S}$ are the $\theta$-braided group (here the braided group is super-Hopf algebra) structure maps \cite{KaDaHa, KaDaHa2} of the Relative Parabose Set $P_{BF}$.
\end{proof}

\section{Fock-like representations of $P_{BF}^{(1,1)}$}  \selabel{singleFockspace}

The relations \equref{parab-paraf}, \equref{parabparaf}, specify the Relative Parabose Set algebra $P_{BF}$. If we take into account all combinations of values of the superscripts $\xi, \eta, \epsilon$ then \equref{parab-paraf} produce 12(=6+6) relations while \equref{parabparaf} produce 28 relations. These can all be found written explicitly in \cite{KaDaHa}.

In this section we are going to study the Relative Parabose Set algebra described by relations \equref{parab-paraf}, \equref{parabparaf} but for the special case for which the $i,j,k$ indices take only a single value each: $i = j = k = 1$. So our algebra will be called the Relative Parabose Set algebra in a sigle parabosonic ($m = 1$) and a single parafermionic ($n = 1$) degree of freedom. It will be denoted $P_{BF}^{(1,1)}$. Its four generators are $b^{+}, b^{-}$ (corresponding to the parabosonic degree of freedom) and $f^{+}, f^{-}$ (corresponding to the parafermionic degree of freedom). The relations satisfied by these generators are explicitly
\begin{equation} \eqlabel{mixed}
\begin{array} {ccc}
\big[ \{ b^{+}, b^{+} \}, f^{-} \big] = 0,  & & \big[ [ f^{+}, f^{-} ], b^{-} \big] = 0     \\
\big[ \{ b^{-}, b^{-} \}, f^{-} \big] = 0,  & &   \big[ \{ b^{+}, b^{-} \}, f^{-} \big] = 0    \\
   &     &              \\
\big[ \{ f^{-}, b^{+} \}, b^{-} \big] = -2 f^{-},  &  &  \{ \{ b^{-}, f^{+} \}, f^{-} \} = 2 b^{-}   \\
\big[ \{ b^{-}, f^{-} \}, b^{+} \big] =  2 f^{-},  & & \{ \{ f^{-}, b^{-} \}, f^{+} \} =  2 b^{-} \\
\big[ \{ b^{-}, b^{+} \}, f^{+} \big] = 0,  & &  \big[ [ f^{-}, f^{+} ], b^{+} \big] = 0    \\
\big[ \{ f^{+}, b^{-} \}, b^{+} \big] = 2 f^{+},  & &  \{ \{ b^{+}, f^{-} \}, f^{+} \} = 2 b^{+}  \\
\big[ \{ b^{+}, f^{+} \}, b^{-} \big] = -2 f^{+},  & &   \{ \{ f^{+}, b^{+} \}, f^{-} \} = 2 b^{+}  \\
   &     &              \\
\big[ \{ f^{-}, b^{-} \}, b^{-} \big] = 0,   &  &    \big[ \{ f^{-}, b^{+} \}, b^{+} \big] = 0   \\
\big[ \{ b^{+}, b^{+} \}, f^{+} \big] = 0,   &  &   \big[ \{ b^{-}, b^{-} \}, f^{+} \big] = 0  \\
\big[ \{ f^{+}, b^{+} \}, b^{+} \big] = 0,   & &    \big[ \{ f^{+}, b^{-} \}, b^{-} \big] = 0   \\
   &     &              \\
\{ \{ b^{-}, f^{-} \}, f^{-} \} = 0,   &  &  \{ \{ b^{-}, f^{+} \}, f^{+} \} = 0  \\
\{ \{ b^{+}, f^{+} \}, f^{+} \} = 0,   &  &  \{ \{ b^{+}, f^{-} \}, f^{-} \} = 0
\end{array}
\end{equation}
The above relations \equref{mixed} stem directly from \equref{parabparaf} and are exactly the same relations with relations (3), p.964 of ref. \cite{Ya2}. Furthermore we also have the ``unmixed relations''
\begin{equation}   \eqlabel{pure}
\begin{array}{ccc}
\big[ b^{-}, \{ b^{+}, b^{-} \} \big] = 2 b^{-}, &  &  \big[ b^{+}, \{ b^{+}, b^{+} \} \big]= 0 \\
\big[ b^{-}, \{ b^{-}, b^{-} \} \big]= 0,  &  & \big[ b^{-}, \{ b^{+}, b^{+} \} \big] =  4 b^{+}  \\
\big[ b^{+}, \{ b^{-}, b^{-} \} \big] = - 4 b^{-} ,  &   & \big[ f^{-}, [ f^{+}, f^{-} ] \big] = 2 f^{-}  \\
\big[ b^{+}, \{ b^{-}, b^{+} \} \big] = - 2 b^{+},  &  &  \big[ f^{+}, [ f^{-}, f^{+} ] \big] = 2 f^{+}
\end{array}
\end{equation}
Relations \equref{pure} stem directly from \equref{parab-paraf} and are exactly the same as relations (1), (2), p.964 of ref. \cite{Ya2}.

So the Relative Parabose Set algebra in a sigle parabosonic ($m = 1$) and a single parafermionic ($n = 1$) degree of freedom will be denoted $P_{BF}^{(1,1)}$, has the 4 generators $b^{+}, b^{-}, f^{+}, f^{-}$ and relations \equref{mixed}, \equref{pure}. These are 32 relations totally.

Before proceeding to reviewing the result of \cite{Ya2}, let us recall a result which was already conjectured from the beginnings of the study of paraparticle algebras (see \cite{GreeMe, OhKa}).
\begin{proposition} \prlabel{uniquenessofFocksp}
If we consider representations of $P_{BF}^{(1,1)}$, under the adjointness conditions $(b^{-})^{\dagger} = b^{+}$ and $(f^{-})^{\dagger} = f^{+}$, on a complex, Euclidean\footnote{Euclidean or pre-Hilbert space, in the sense that it is an inner product space, but not necessarily complete (with respect to the inner product)}, infinite dimensional space possessing a unique vacuum vector $|0 \rangle$ satisfying $b^{-} |0 \rangle = f^{-} |0 \rangle = 0$, then the following conditions ($p$ may be an arbitrary positive integer)
\begin{equation}  \eqlabel{singleoutFock}
\begin{array}{ccccc}
b^{-} b^{+} |0 \rangle = f^{-} f^{+} |0 \rangle = p |0 \rangle  &  &   &  &
b^{-} f^{+} |0 \rangle = f^{-} b^{+} |0 \rangle = 0
\end{array}
\end{equation}
single out an irreducible representation which is unique up to unitary equivalence.
\end{proposition}
In other words the above proposition tells us that for any positive integer $p$ there is an irreducible representation of $P_{BF}^{(1,1)}$ uniquely specified (up to unitary equivalence) by $b^{-} |0 \rangle = f^{-} |0 \rangle = 0$ together with relations \equref{singleoutFock}. It is these representations which we will call \textbf{Fock-like representations of $\mathbf{P_{BF}^{(1,1)}}$} from now on. We emphasize on the fact that each one of \textbf{these representations is characterized by the positive integer} $\mathbf{p}$, in other words the value of $p$ is part of the data which uniquely specify the representation.

Finally, we introduce the following notation ($\epsilon, \eta = \pm$)
\begin{equation}  \eqlabel{defoper}
\begin{array}{cccccccccc}
N_{b} = \frac{1}{2} \{ b^{+}, b^{-} \} - \frac{p}{2} , &  N_{f} = \frac{1}{2} [ f^{+}, f^{-} ] + \frac{p}{2} , & R^{\eta} = \frac{1}{2} \{ b^{\eta}, f^{\eta} \}  , & Q^{\epsilon} = \frac{1}{2} \{ b^{-\epsilon}, f^{\epsilon} \}
\end{array}
\end{equation}

In \cite{Ya2}, the authors investigate the structure of the carrier spaces of the Fock-like representations of $P_{BF}^{(1,1)}$. Their results may be summarized in the following (notation due to us)
\begin{theorem}[\textbf{Fock-spaces structure of} $\mathbf{P_{BF}^{(1,1)}}$] \thlabel{FockspstructPBF11}
The carrier spaces of the Fock-like representations of $P_{BF}^{(1,1)}$, uniquely determined (as representations) under the conditions specified in the above conjecture, are:
\begin{equation}  \eqlabel{carrierFocksp}
\bigoplus_{n=0}^{p} \bigoplus_{m=0}^{\infty} \mathcal{V}_{m,n}
\end{equation}
where $p$ is an arbitrary (but fixed) positive integer and the subspaces $\mathcal{V}_{m,n}$ are 2-dimensional except for the cases $m = 0$, $n = 0, p$ i.e. except the subspaces $\mathcal{V}_{0,n}$, $\mathcal{V}_{m,0}$, $\mathcal{V}_{m,p}$ which are 1-dimensional. Let us see how the corresponding vectors look like:

$\blacktriangleright$ \underline{If $1 \leq m$ and $1 \leq n < p$}, then the subspace $\mathcal{V}_{m,n}$ is spanned by all vectors of the form
\begin{equation}   \eqlabel{spanvect}
\Big| \begin{array}{c}
      m_{1}, m_{2}, ..., m_{l}  \\
n_{0},n_{1}, n_{2}, ..., n_{l}
      \end{array} \Big\rangle \equiv (f^{+})^{n_{0}} (b^{+})^{m_{1}} (f^{+})^{n_{1}} (b^{+})^{m_{2}} (f^{+})^{n_{2}} ... (b^{+})^{m_{l}} (f^{+})^{n_{l}} |0 \rangle
\end{equation}
where $m_{1} + m_{2} + ... + m_{l} = m$ , $n_{0} + n_{1} + n_{2} + ... + n_{l} = n$ and $m_{i} \geq 1$ (for $i = 1, 2, ... , l$), $n_{i} \geq 1$ (for $i = 1, 2, ..., l-1$) and $n_{0}, n_{l} \geq 0$.

For any specific combination of values $(m,n)$ \textbf{the corresponding subspace} $\mathbf{\mathcal{V}_{m,n}}$ \textbf{has a basis consisting of the two vectors}
{\small
\begin{equation} \eqlabel{subspbase}
\begin{array}{c}
\phi_{m,n}\equiv|m,n,\alpha\rangle\equiv(f^{+})^{n} (b^{+})^{m}|0\rangle=\Big| \begin{array}{c}
                                                                               m  \\
                                                                               n, 0
                                                                               \end{array} \Big\rangle   \\
\textrm{and}  \\
\psi_{m,n}\equiv|m,n,\beta\rangle\equiv (f^{+})^{(n-1)} (b^{+})^{(m-1)} R^{+}|0\rangle=\frac{1}{2} \Big| \begin{array}{c}
                                                                                                         m  \\
                                                                                                         n-1, 1
                                                                                                         \end{array}\Big\rangle +
\frac{1}{2} \Big| \begin{array}{c}
                  m-1, 1  \\
                  n-1, 1, 0
                  \end{array} \Big\rangle
\end{array}
\end{equation}
}
where we use the notation $R^{\eta} = \frac{1}{2} \{ b^{\eta}, f^{\eta} \}$ for $\eta = \pm$. In other words we can always express any vector $\Big| \begin{array}{c}
      m_{1}, m_{2}, ..., m_{l}  \\
n_{0},n_{1}, n_{2}, ..., n_{l}
      \end{array} \Big\rangle$
of the form \equref{spanvect} as a linear combination of vectors of the form \equref{subspbase}
{\small
\begin{equation}
\Big| \begin{array}{c}
      m_{1}, m_{2}, ..., m_{l}  \\
n_{0},n_{1}, n_{2}, ..., n_{l}
      \end{array} \Big\rangle=c_{1}|m,n,\alpha\rangle+c_{2}|m,n,\beta\rangle=c_{1}\phi_{m,n}+c_{2}\psi_{m,n}
\end{equation}
}

$\blacktriangleright$ \underline{If $m=0$ or $n=0$ or $n=p$}, the vectors $|0, n, \beta \rangle$ and $|m, 0, \beta \rangle$ are (by definition) zero and furthermore if $m\neq0$ the vector $|m, p, \beta \rangle$ becomes proportional to $|m, p, \alpha \rangle$, thus:
{\small
\begin{equation}
\psi_{0,n}\equiv|0,n,\beta\rangle=\psi_{m,0}\equiv|m,0,\beta\rangle=0 \ and \ \psi_{m,p}\equiv|m,p,\beta\rangle=\frac{1}{p}\phi_{m,p}\equiv\frac{1}{p}|m,p,\alpha\rangle
\end{equation}
}
Consequently, the corresponding subspaces $\mathcal{V}_{0,n}$, $\mathcal{V}_{m,0}$, $\mathcal{V}_{m,p}$ are 1-dimensional and their bases consist of the single vectors $\phi_{0,n}\equiv|0,n,\alpha\rangle$, $\phi_{m,0}\equiv|m,0,\alpha\rangle$, $\phi_{m,p}\equiv|m,p,\alpha\rangle$ respectively.

$\blacktriangleright$ \underline{If $n \geq p+1$}, all basis vectors of \equref{subspbase} vanish.
\end{theorem}
The above described subspaces of $\bigoplus_{n=0}^{p} \bigoplus_{m=0}^{\infty} \mathcal{V}_{m,n}$ can be visualized as:
{\small
\begin{equation}  \eqlabel{2dimladdersubsp}
\begin{array}{cccccccc}
    \langle\phi_{0,0}\rangle & \langle\phi_{0,1}\rangle & \ldots  &  \langle\phi_{0,n}\rangle & \ldots & \ldots & \langle\phi_{0,p-1}\rangle & \langle\phi_{0,p}\rangle  \\
    \langle\phi_{1,0}\rangle & \langle\begin{array}{c}\phi_{1,1} \\ \psi_{1,1}\end{array}\rangle & \ldots & \langle\begin{array}{c}\phi_{1,n} \\ \psi_{1,n}\end{array}\rangle & \ldots &  \ldots & \langle\begin{array}{c}\phi_{1,p-1} \\ \psi_{1,p-1}\end{array}\rangle & \langle\phi_{1,p}\rangle  \\
    \vdots & \vdots & \ldots  & \vdots & \ldots & \ldots  & \vdots & \vdots     \\
    \langle\phi_{m,0}\rangle & \langle\begin{array}{c}\phi_{m,1} \\ \psi_{m,1}\end{array}\rangle & \ldots & \langle\begin{array}{c}\phi_{m,n} \\ \psi_{m,n}\end{array}\rangle & \langle\begin{array}{c}\phi_{m,n+1} \\ \psi_{m,n+1}\end{array}\rangle & \ldots & \ldots & \langle\phi_{m,p}\rangle \\
    \vdots & \vdots &  \ldots & \langle\begin{array}{c}\phi_{m+1,n} \\ \psi_{m+1,n}\end{array}\rangle & \ldots &  \ldots  & \vdots  & \vdots     \\
    \vdots & \vdots & \ldots & \vdots & \ldots & \ldots  & \vdots & \vdots
\end{array}
\end{equation}
}
We now present the formulae describing explicitly the action of the generators (and hence of the whole algebra) of $P_{BF}^{(1,1)}$ on the carrier spaces $\bigoplus_{n=0}^{p} \bigoplus_{m=0}^{\infty}\mathcal{V}_{m,n}$ for any positive integer $p$:
{\small
\begin{equation}  \eqlabel{equat3}
\begin{array}{l}
\bullet b^{-}\cdot\phi_{m,n} = \left\{%
\begin{array}{ll}
(-1)^{n}m \phi_{m-1,n}-2(-1)^{n}nm \psi_{m-1,n}, \ \underline{m:even}   \\   \\
-(-1)^{n}\big(2n-m-(p-1)\big) \phi_{m-1,n}-2(-1)^{n}n(m-1) \psi_{m-1,n}, \  \underline{m:odd}
\end{array}
\right.  \\
\bullet b^{-}\cdot\psi_{m,n}=\left\{%
\begin{array}{ll}
-(-1)^{n} \phi_{m-1,n}+(-1)^{n}\big(2n-m-p\big) \psi_{m-1,n}, \  \underline{m:even}   \\   \\
-(-1)^{n} \phi_{m-1,n}-(-1)^{n}(m-1) \psi_{m-1,n}, \  \underline{m:odd}
\end{array}
\right.    \\
  \\
\bullet f^{-}\cdot\phi_{m,n}=n(p+1-n) \phi_{m,n-1}, \qquad \qquad \bullet b^{+}\cdot\phi_{m,n} = (-1)^{n} \phi_{m+1,n} - (-1)^{n}2n \psi_{m+1,n}  \\
\bullet f^{-}\cdot\psi_{m,n}=\phi_{m,n-1}+(n-1)(p-n)\psi_{m,n-1},  \quad \bullet b^{+}\cdot\psi_{m,n} = -(-1)^{n} \psi_{m+1,n} \\  \\
\!\!
\begin{array}{lll}
\bullet f^{+}\!\cdot\!\phi_{m,n} = \left\{%
\begin{array}{cc}
\phi_{m,n+1},  \underline{if \ n \leq p-1}   \\
0, \qquad \qquad \underline{if \ n \geq p}
\end{array}
\right.
   &    &
\bullet f^{+}\!\cdot\!\psi_{m,n} = \left\{%
\begin{array}{cc}
\psi_{m,n+1},  \underline{if \ n \leq p-1}   \\
0, \qquad \qquad \underline{if \ n \geq p}
\end{array}
\right.
\end{array}
\end{array}
\end{equation}
} for all integers $0 \leq m$, $0 \leq n \leq p$.

The direct proof \cite{KaDaHa3} of these formulae involves lengthy ``normal--ordering" algebraic computations inside $P_{BF}^{(1,1)}$. From these formulae, we can easily figure out that we have a kind of \textit{generalized creation--annihilation operators} acting on a
two dimensional ladder of subspaces. Initiating from the remark that we are dealing with a cyclic module which moreover can be generated by any of its elements, we can prove \cite{KaDaHa3} that the Fock--like representations, explicitly given by \equref{2dimladdersubsp}, \equref{equat3} and visually represented in the above diagram, are irreducible representations (or: simple $P_{BF}^{(1,1)}$-modules) for any $p \in \mathbb{N}^{*}$.

Furthermore, defining $\verb"deg"\varphi_{m,n}\!=\!\verb"deg"\psi_{m,n}\!=\!\big(m\ \textsf{mod}\ 2,\ n\ \textsf{mod}\ 2 \big)\!\in\!\mathbb{Z}_{2}\!\times\!\mathbb{Z}_{2}$ for the carrier spaces and $\verb"deg" b^{\pm}\!=\!(1,0)$, $\verb"deg" f^{\pm}\!=\!(0,1)$ for the algebra, the Fock--like representations of $P_{BF}^{(1,1)}$ over $\bigoplus_{n=0}^{p} \bigoplus_{m=0}^{\infty}\mathcal{V}_{m,n}$, become $(\mathbb{Z}_{2} \times\mathbb{Z}_{2})$--graded modules, $\forall p\in\mathbb{N}^{*}$. (For proof and details see \cite{KaDaHa3}). The above described $\mathbb{Z}_{2}$ grading for the carrier spaces $\bigoplus_{n=0}^{p} \bigoplus_{m=0}^{\infty}\mathcal{V}_{m,n}$ can be visually represented in the following figure
{\scriptsize
\begin{equation} \eqlabel{carrspZ2grad}
\xymatrix{
\stackrel{(0,0)}{\bullet} & \stackrel{(0,1)}{\bullet} & \stackrel{(0,0)}{\bullet} & \stackrel{(0,1)}{\bullet} & \stackrel{(0,0)}{\bullet} & \cdots & \cdots & \stackrel{(0, \ p \textrm{mod} 2)}{\bullet}  \\
\stackrel{(1,0)}{\bullet} & \stackrel{(1,1)}{\blacksquare} & \stackrel{(1,0)}{\blacksquare} & \stackrel{(1,1)}{\blacksquare} & \stackrel{(1,0)}{\blacksquare} & \cdots & \cdots & \stackrel{(1, \ p \textrm{mod} 2)}{\bullet}  \\
\stackrel{(0,0)}{\bullet} & \stackrel{(0,1)}{\blacksquare} & \stackrel{(0,0)}{\blacksquare} & \stackrel{(0,1)}{\blacksquare} & \stackrel{(0,0)}{\blacksquare} & \cdots & \cdots & \stackrel{(0, \ p \textrm{mod} 2)}{\bullet}  \\
\stackrel{(1,0)}{\bullet} & \stackrel{(1,1)}{\blacksquare} & \stackrel{(1,0)}{\blacksquare} & \stackrel{(1,1)}{\blacksquare} & \stackrel{(1,0)}{\blacksquare} & \cdots & \cdots & \stackrel{(1, \ p \textrm{mod} 2)}{\bullet}  \\
\stackrel{(0,0)}{\bullet} & \stackrel{(0,1)}{\blacksquare} & \stackrel{(0,0)}{\blacksquare} & \stackrel{(0,1)}{\blacksquare} & \stackrel{(0,0)}{\blacksquare} & \cdots & \cdots & \stackrel{(0, \ p \textrm{mod} 2)}{\bullet}  \\
\vdots &  \vdots  & \vdots  & \vdots  & \vdots  & \ddots & \ddots & \vdots
}
\end{equation}
}
where we have denoted with a bullet $``\bullet"$ the $\mathcal{V}_{m,n}$ subspaces for which either $m=0$ or $n=0$ or $n=p$ i.e. the $1$-dim subspaces for the first and the $p$-th columns and the first row, while we have denoted by a black square $``\blacksquare"$ the $\mathcal{V}_{m,n}$ subspaces for which $m\neq0$ and $n\neq0$ and $n\neq p$ i.e. the $2$-dim subspaces which lie in the ``internal" rows and columns of the above ``matrix" of subspaces. We must underline at this point, that the $\mathbb{Z}_{2}$-grading for the $P_{BF}$ algebra described above is inequivalent to the $\mathbb{Z}_{2}$-grading introduced by \cite{Ya1} and described in \prref{gradedstruct}. These are inequivalent in  the sense that they correspond to different $(\mathbb{Z}_{2}\times\mathbb{Z}_{2})$-actions on the carrier space which is $P_{BF}$ itself. In other words we have the same carrier space but with different $(\mathbb{Z}_{2}\times\mathbb{Z}_{2})$-module structures. The advantage of the former choice \cite{Ya1} lies in the fact that it provides us with the isomorphism $P_{BF} \cong \mathbb{U}(L_{\mathbb{Z}_{2} \times \mathbb{Z}_{2}})$ whereas the advantage of the present choice lies in the fact that it provides us with the structure of a $(\mathbb{Z}_{2}\times\mathbb{Z}_{2})$-module (in combination of course with the carrier space grading described in \equref{carrspZ2grad}).

\textbf{Note 1:} Inside each one of the above presented $2$--dim subspaces $\mathcal{V}_{m,n}=\big\langle\begin{array}{c}\phi_{m,n} \\ \psi_{m,n}\end{array}\big\rangle$ (with $m\neq0$, $n\neq0,p$) the vectors $\phi_{m,n}$, $\psi_{m,n}$ of \equref{subspbase} are linearly independent and constitute a basis. However, these vectors are neither orthogonal nor normalized. We can show that an orthonormal set of basis vectors can be obtained as
$$
\begin{array}{cccc}
| m, n, + \rangle = c_{+} | m, n, \alpha \rangle &  &  & | m, n, - \rangle = -c_{-} \Big( | m, n, \alpha \rangle - p | m, n, \beta \rangle \Big)
\end{array}
$$
where $c_{\pm}$ are suitable normalization factors \cite{Ya2}. Now we can show the orthonormalization relation $\langle m,n,s|m^{'},n^{'},s^{'}\rangle=\delta_{m,m^{'}}\delta_{n,n^{'}}\delta_{s,s^{'}}$ and the completeness relation $\sum_{m=0}^{\infty} \sum_{n=0}^{p} \sum_{s=\pm}|m,n,s\rangle \langle m,n,s| = 1$.

\textbf{Note 2:} If we consider the Hermitian operators $N_{s} = \frac{1}{p} \big( N_{f}^{2} -(p+1)N_{f} + f^{+}f^{-} + \frac{p}{2} \big)$, $N_{f}=\frac{1}{2}[f^{+},f^{-}]+\frac{p}{2}$ and $N_{b} = \frac{1}{2}\{b^{+},b^{-} \}-\frac{p}{2}$ we can show that they constitute a \textit{Complete Set of Commuting Observables} (\textit{C.S.C.O.}): We have $[N_{b}, N_{f}]=[N_{b}, N_{s}]=[N_{f}, N_{s}]=0$; their common eigenvectors are exactly the elements of the orthonormal basis formerly described. Any vector $|m,n,s\rangle$ is uniquely determined as an eigenvector of $N_{b}$, $N_{f}$, $N_{s}$ by its eigenvalues $0 \leq m$, $0 \leq n \leq p$ and $s=\pm\frac{1}{2}$ respectively.

\section{Constructing infinite dimensional, decomposable super-representations for a Lie superalgebra}    \selabel{Focklikemodule}

In this section we will build infinite dimensional representations for an arbitrary Lie superalgebra starting from the Fock-like representations of $P_{BF}^{(1,1)}$ (described in \thref{FockspstructPBF11}) and the paraparticle realizations of an arbitrary Lie superalgebra (described in \prref{parapartrealiz}).

\paragraph{$\bullet$ \textbf{Preliminaries:}}
Before we start let us recall a lemma which the will be the link between the above ingredients. Let $A$ and $B$ be two (complex, associative) algebras and $\phi: A \rightarrow B$ an (associative) algebra homomorphism between them.
\begin{lemma} \lelabel{Parabreprfrombosrepr}
For any vector space $V$, any ${}_{B}V$ module of the algebra $B$ immediately gives rise to a $ \ {}_{A}V$ module of the algebra $A$ through the homomorphism $\phi: A \rightarrow B$
\begin{equation} \eqlabel{Parabactfrbosact}
a \cdot \vec{v} \stackrel{def.}{=} \phi(a) \cdot \vec{v} \equiv b \cdot \vec{v}
\end{equation}
where $\phi(a) = b$ and $\vec{v}$ stands for any element of $ \ V$. In the lhs of the above ``$\cdot$" stands for the $A$-action on $ \ V$, while in the rhs ``$\cdot$" stands for the $B$-action on $ \ V$.

Furthermore, if ${}_{B}V$ is an irreducible representation (a simple module) of $B$ and $\phi$ is an epimorphism of algebras then ${}_{A}V$ will also be an irreducible representation (a simple module) of $A$.
\end{lemma}

\paragraph{$\bullet$ \textbf{Working in the Fock-space of $\mathbf{P_{BF}^{(1,1)}}$:}} We will start by calculating the action of the $Q^{-}$ and $Q^{+}$ operators (defined by rel. \equref{defoper}) on the basis elements of the carrier space  $\bigoplus_{n=0}^{p} \bigoplus_{m=0}^{\infty} \mathcal{V}_{m,n}$ (described in \thref{FockspstructPBF11}).

Let us recall the relations:
\begin{equation}  \eqlabel{1}
\begin{array}{ccccccc}
\big[ b^{-}, Q^{-} \big] = f^{-}  & , & \big[ b^{-}, Q^{+} \big] = 0  & , &  \{ f^{-}, Q^{+} \} = b^{-} & , & \{ f^{-}, Q^{-} \} = 0 \\
\big[ b^{+}, Q^{+} \big] = -f^{+} & , & \big[ b^{+}, Q^{-} \big] = 0  & , &  \{ f^{+}, Q^{-} \} = b^{+} & , & \{ f^{+}, Q^{+} \} = 0
\end{array}
\end{equation}
which are clearly a rewriting of the corresponding relations from \equref{mixed}. Based on relations \equref{1} we have the following
\begin{lemma}
We have the following relations
\begin{equation}   \eqlabel{2}
\begin{array}{c}
Q^{-} (f^{+})^{n} = (-f^{+})^{n} Q^{-} + \sum_{i=1}^{n} (-f^{+})^{i-1} b^{+} (f^{+})^{n-i}  \\
b^{+} (f^{+})^{k} = (-f^{+})^{k} b^{+} + 2 k R^{+} (f^{+})^{k-1}    \\
\end{array}
\end{equation}
for all integers $n, k \geq 0$, and
\begin{equation}    \eqlabel{3}
Q^{-} (f^{+})^{m} = (-f^{+})^{m} Q^{-} + m (-f^{+})^{m-1} b^{+} + m(m-1) (-f^{+})^{m-2} R^{+}
\end{equation}
for all positive integers $m \geq 0$.
\end{lemma}
\begin{proof}
For each one of \equref{2} induction on $n$, $k$ respectively.   \\
Their combination then yields \equref{3}.
\end{proof}
Utilizing the above lemma together with the relation $\{Q^{-},R^{+}\}=(b^{+})^{2}$ we get the action of the $Q^{-}$ operator on the basis elements of $\bigoplus_{n=0}^{p} \bigoplus_{m=0}^{\infty} \mathcal{V}_{m,n}$
\begin{proposition}    \prlabel{Q-action}
We have the following formulas
\begin{equation}    \eqlabel{4}
\begin{array}{c}
\mathbf{Q^{-}|m,n,\alpha\rangle=(-1)^{n-1}n|m+1,n-1,\alpha\rangle+(-1)^{n}n(n-1)|m+1,n-1,\beta\rangle} \\
  \textrm{(for } 0\leq m \textrm{ and } 0\leq n\leq p \textrm{)}   \\  \textrm{and:} \\
\mathbf{Q^{-}|m,n,\beta\rangle=(-1)^{n-1}|m+1,n-1,\alpha\rangle+(-1)^{n}(n-1)|m+1,n-1,\beta\rangle} \\
  \textrm{(for } 1\leq m \textrm{ and } 1\leq n\leq p \textrm{)}
\end{array}
\end{equation}
Note that the above action-formulae imply
\begin{equation}  \eqlabel{41}
\begin{array}{ccccc}
\mathbf{Q^{-}|m,n,\alpha\rangle = nQ^{-}|m,n,\beta\rangle} & & and & & \mathbf{\big(Q^{-}\big)^{2}=0}
\end{array}
\end{equation}
verifying thus the (expected) nilpotency in the action of $Q^{-}$.
\end{proposition}
\begin{proof}
Let us start with the action of $Q^{-}$ on the $\varphi_{m,n}\equiv|m,n,\alpha\rangle$ vectors:
$$
\begin{array}{c}
Q^{-}|m,n,\alpha\rangle=\underbrace{Q^{-}(f^{+})^{n}}(b^{+})^{m}|0\rangle= \\
    \\
=\underbrace{\big( (-f^{+})^{n}Q^{-}+n(-f^{+})^{n-1}b^{+}+n(n-1)(-f^{+})^{n-2}R^{+} \big)}(b^{+})^{m}|0\rangle=\\
    \\
=(-f^{+})^{n}Q^{-}(b^{+})^{m}|0\rangle+n(-f^{+})^{n-1}(b^{+})^{m+1}|0\rangle+n(n-1)(-f^{+})^{n-2}R^{+}(b^{+})^{m}|0\rangle= \\
    \\
\textrm{but taking into account that } \big[b^{+},Q^{-}\big]=0=\big[b^{+},R^{+}\big] \textrm{ and } Q^{-}|0\rangle=0 \textrm{ we get} \\
       \\
=n(-f^{+})^{n-1}(b^{+})^{m+1}|0\rangle+n(n-1)(-f^{+})^{n-2}(b^{+})^{m}R^{+}|0\rangle= \\
    \\
(-1)^{n-1}n|m+1,n-1,\alpha\rangle+(-1)^{n}n(n-1)|m+1,n-1,\beta\rangle
\end{array}
$$
for $m\geq1$, $n\geq1$.
A direct computation also yields that $Q^{-}|m,0,\alpha\rangle=0$ for $m\geq0$ and $Q^{-}|0,n,\alpha\rangle=(-1)^{n-1}n|1,n-1,\alpha \rangle+(-1)^{n}n(n-1)|1,n-1,\beta\rangle$ for $n\geq0$ which completes the proof of the first action formulae of \equref{4}.

Regarding now the second formulae of \equref{4}, i.e. the action of $Q^{-}$ on the $\psi_{m,n}\equiv|m,n,\beta\rangle$ vectors we get (considering
that $m\geq1$ and $n\geq2$)
$$
\begin{array}{c}
Q^{-}|m,n,\beta\rangle=\underbrace{Q^{-}(f^{+})^{n-1}}(b^{+})^{m-1}R^{+}|0\rangle=  \\
  \\
=\underbrace{\big( (-f^{+})^{n-1}Q^{-}+(n-1)(-f^{+})^{n-2}b^{+}+(n-1)(n-2)(-f^{+})^{n-3}R^{+} \big)}(b^{+})^{m-1}R^{+}|0\rangle= \\
    \\
=(-f^{+})^{n-1}Q^{-}(b^{+})^{m-1}R^{+}|0\rangle+(n-1)(-f^{+})^{n-2}(b^{+})^{m}R^{+}|0\rangle+  \\
(n-1)(n-2)(-f^{+})^{n-3}R^{+}(b^{+})^{m-1}R^{+}|0\rangle=  \\
   \\
\textrm{using again } \big[b^{+},R^{+}\big]=0=\big[b^{+},Q^{-}\big]   \\
   \\
=(-f^{+})^{n-1}(b^{+})^{m-1}Q^{-}R^{+}|0\rangle+(n-1)(-f^{+})^{n-2}(b^{+})^{m}R^{+}|0\rangle+  \\
(n-1)(n-2)(-f^{+})^{n-3}(b^{+})^{m-1}(R^{+})^{2}|0\rangle=   \\
   \\
\textrm{applying the nilpotency} (R^{+})^{2}=(\frac{1}{2}\{b^{\eta},f^{\eta}\})^{2}=0 \textrm{ (which can be directly }  \\
\textrm{ deduced from \equref{mixed}) we get the last summand of the above to be } 0 \textrm{ thus } \\
\end{array}
$$
$$
\begin{array}{c}
=(-f^{+})^{n-1}(b^{+})^{m-1}Q^{-}R^{+}|0\rangle+(n-1)(-f^{+})^{n-2}(b^{+})^{m}R^{+}|0\rangle=  \\
   \\
\textrm{ At this point we recall the relation } \{Q^{-},R^{+}\}=(b^{+})^{2} \textrm{ and } Q^{-}|0\rangle=0 \textrm{ thus}   \\
    \\
=(-f^{+})^{n-1}(b^{+})^{m-1}\big(-R^{+}Q^{-}+(b^{+}\big)^{2}\big)|0\rangle+(n-1)(-f^{+})^{n-2}(b^{+})^{m}R^{+}|0\rangle=  \\
  \\
=(-1)^{n-1}(f^{+})^{n-1}(b^{+})^{m+1}|0\rangle+(n-1)(-1)^{n}(f^{+})^{n-2}(b^{+})^{m}R^{+}|0\rangle=  \\
  \\
=(-1)^{n-1}|m+1,n-1,\alpha\rangle+(-1)^{n}(n-1)|m+1,n-1,\beta\rangle
\end{array}
$$
$m\geq1$ and $n\geq2$. Finally, a direct computation yields $Q^{-}|m,1,\beta\rangle=|m+1,0,\alpha\rangle$ for the $n=1$ case which completes the proof of the second one of relations \equref{4}. Recall that the $\psi_{m,n}\equiv|m,n,\beta\rangle$ vectors vanish for either $m=0$ or $n=0$.
\end{proof}
Again using \equref{1} we get the following
\begin{lemma}
We have the following relations
\begin{equation}   \eqlabel{5}
Q^{+} (b^{+})^{m} = (b^{+})^{m} Q^{+} + \sum_{i=1}^{m} (b^{+})^{i-1} f^{+} (b^{+})^{m-i}
\end{equation}
for all positive integers $m \geq 1$ and
\begin{equation}   \eqlabel{6}
Q^{+} (b^{+})^{m} = (b^{+})^{m} Q^{+} + \frac{( 1 - (-1)^{m})}{2} f^{+} (b^{+})^{m-1} + 2 \big( \lfloor \frac{m-2}{2} \rfloor + 1 \big) (b^{+})^{m-2} R^{+}
\end{equation}
for all positive integers $m \geq 2$. The notation $\lfloor x \rfloor$ stands for the integer part of $x$.
\end{lemma}
\begin{proof}
\equref{5} can be proved by induction on $m$.   \\
\equref{6} is shown using \equref{5} and $[(b^{+})^{2},f^{+}]=0$ (which can be found among relations \equref{mixed}):
$$
\begin{array}{c}
Q^{+}(b^{+})^{m}=(b^{+})^{m}Q^{+}+\displaystyle\sum_{i=1}^{m}(b^{+})^{i-1}f^{+}(b^{+})^{m-i}= \\
\textrm{Considering } m\geq2 \textrm{ from now on we get} \\
=(b^{+})^{m}Q^{+}+f^{+}(b^{+})^{m-1}+\displaystyle\sum^m_{\substack{i=2 \\
                                                       i:even}}\underbrace{b^{+}f^{+}}(b^{+})^{m-2}+\sum_{\substack{i=2 \\
                                                                                                       i:odd}}^{m}f^{+}(b^{+})^{m-1}= \\
=(b^{+})^{m}Q^{+}+f^{+}(b^{+})^{m-1}+\displaystyle\sum^m_{\substack{i=2 \\
                                                       i:even}}(\underbrace{2R^{+}-f^{+}b^{+}})(b^{+})^{m-2}+\sum_{\substack{i=2 \\
                                                                                                       i:odd}}^{m}f^{+}(b^{+})^{m-1}= \\
=(b^{+})^{m}Q^{+}+f^{+}(b^{+})^{m-1}+2\Big(\displaystyle\sum^m_{\substack{i=2 \\
                                                                          i:even}}1\Big)R^{+}(b^{+})^{m-2}-\displaystyle\sum^m_{\substack{i=2 \\
                                                                                           i:even}}f^{+}(b^{+})^{m-1}
                                                                                           +\sum_{\substack{i=2 \\
                                                                                                            i:odd}}^{m}f^{+}(b^{+})^{m-1}= \\
=(b^{+})^{m}Q^{+}+f^{+}(b^{+})^{m-1}+2\big(\lfloor\frac{m-2}{2}\rfloor+1\big)R^{+}(b^{+})^{m-2}-\displaystyle\sum_{i=2}^{m}(-1)^{i}f^{+}(b^{+})^{m-1}= \\
=(b^{+})^{m}Q^{+}+\Big(1-\displaystyle\sum_{i=2}^{m}(-1)^{i}\Big)f^{+}(b^{+})^{m-1}+2\big(\lfloor\frac{m-2}{2}\rfloor+1\big)R^{+}(b^{+})^{m-2}  \\
\textrm{finally taking into account that } \big[b^{+},R^{+}\big]=0 \textrm{ we get } \\
   \\
Q^{+}(b^{+})^{m}=(b^{+})^{m}Q^{+}+\frac{( 1 - (-1)^{m})}{2}f^{+}(b^{+})^{m-1}+2\big(\lfloor\frac{m-2}{2}\rfloor+1\big)(b^{+})^{m-2}R^{+}
\end{array}
$$
\end{proof}
In the above given proof we have used $\displaystyle\sum^m_{\substack{i=2 \\
                                                                          i:even}}1=\lfloor\frac{m-2}{2}\rfloor+1$
where the notation $\lfloor x \rfloor$ stands for the floor function or the integer part of $x$ defined as
$$
\textrm{for any } x\in\mathbb{R} \textrm{ \ } \lfloor x \rfloor \textrm{ is the greatest integer not exceeding } x
$$
Finally, we note that for the special case for which $m=1$, \equref{5} becomes $Q^{+}b^{+}=b^{+}Q^{+}+f^{+}$ which is simply a rewriting of the corresponding relation of \equref{1}.

Utilizing the above relations, together with $\{Q^{+},R^{+}\}=0$ we get the action of the $Q^{+}$ operator on the basis elements of $\bigoplus_{n=0}^{p} \bigoplus_{m=0}^{\infty} \mathcal{V}_{m,n}$
{\small
\begin{proposition}   \prlabel{Q+action}
We have the following relations
\begin{equation}  \eqlabel{7}
\begin{array}{c}
\mathbf{Q^{+}|m,n,\alpha\rangle=\frac{\big((-1)^{n}-(-1)^{n+m}\big)}{2}|m-1,n+1,\alpha\rangle+
2(-1)^{n}\big(\lfloor\frac{m-2}{2}\rfloor+1\big)|m-1,n+1,\beta\rangle}= \\  \\
=\left\{%
\begin{array}{c}
(-1)^nm| m-1,n+1,\beta\rangle, \ \ \ \ \ \ \ \ \ \ \ \ \ \ \ \ \ \ \ \ \ \ \ \ \ \ \ \ \ \ \ \ \ \ \underline{m:even}   \\
(-1)^n\big(|m-1,n+1,\alpha\rangle+(m-1)|m-1,n+1,\beta\rangle\big), \  \underline{m:odd}
\end{array}
\right.
\\ \\  \\
\mathbf{Q^{+}|m,n,\beta\rangle= \frac{\big((-1)^{n-1}-(-1)^{n+m}\big)}{2}|m-1,n+1,\beta\rangle}= \\ \\
=\left\{%
\begin{array}{c}
(-1)^{(n-1)}|m-1,n+1,\beta\rangle, \ \underline{m:even}   \\
0, \ \ \ \ \ \ \ \ \ \ \ \ \ \ \ \ \ \ \ \ \ \ \ \ \ \ \ \ \ \ \ \underline{m:odd}
\end{array}
\right.
\end{array}
\end{equation}
for $0 \leq m$ and $0 \leq n \leq p$. Note that the above action-formulae imply
\begin{equation}  \eqlabel{41}
\mathbf{\big(Q^{+}\big)^{2}=0}
\end{equation}
verifying thus the (expected) nilpotency in the action of $Q^{+}$.
\end{proposition}
}
\begin{proof}
In order to compute the $Q^{+}$ action on the $\phi_{m,n}=|m,n,\alpha\rangle$ vectors, we start by applying $\{Q^{+},f^{+}\}=0$ and then \equref{6} successively:
{\small
$$
\begin{array}{c}
Q^{+}|m,n,\alpha\rangle=\overbrace{Q^{+}(f^{+})^{n}}(b^{+})^{m}|0\rangle=\overbrace{(-1)^{n}(f^{+})^{n}Q^{+}}(b^{+})^{m}|0\rangle=
(-1)^{n}(f^{+})^{n}\underbrace{Q^{+}(b^{+})^{m}}|0\rangle= \\
  \\
=(-1)^{n}(f^{+})^{n}\underbrace{\Big((b^{+})^{m}Q^{+}+\frac{(1-(-1)^{m})}{2}f^{+}(b^{+})^{m-1}+2\big(\lfloor\frac{m-2}{2}\rfloor+1\big)(b^{+})^{m-2} R^{+}\Big)}|0\rangle= \\
   \\
  \textrm{taking into account that } Q^{+}|0\rangle=0 \textrm{ we get}  \\
  \\
=(-1)^{n}\frac{(1-(-1)^{m})}{2}(f^{+})^{n+1}(b^{+})^{m-1}|0\rangle+2(-1)^{n}\big(\lfloor\frac{m-2}{2}\rfloor+1\big)(f^{+})^{n}(b^{+})^{m-2} R^{+}|0\rangle=  \\
   \\
=\frac{((-1)^{n}-(-1)^{n+m})}{2}|m-1,n+1,\alpha\rangle+2(-1)^{n}\big(\lfloor\frac{m-2}{2}\rfloor+1\big)|m-1,n+1,\beta\rangle
\end{array}
$$
}
for $m\geq2$ and $n\geq1$. A direct computation now yields that $Q^{+}|1,n,\alpha\rangle=2(-1)^{n}|0,n+1,\alpha\rangle$ and $Q^{+}|0,n,\alpha\rangle=0$ for $n\geq0$ and $Q^{+}|m,0,\alpha\rangle=\frac{\big(1-(-1)^{m}\big)}{2}|m-1,1,\alpha\rangle+2\big(\lfloor\frac{m-2}{2}\rfloor+1\big)|m-1,1,\beta\rangle$ for $m\geq0$ completing thus the proof of the first of the action formulae \equref{7}.

Regarding now the $Q^{+}$ action on the $\psi_{m,n}=|m,n,\beta\rangle$ vectors, the computation proceeds as follows:
{\small
$$
\begin{array}{c}
Q^{+}|m,n,\beta\rangle=\overbrace{Q^{+}(f^{+})^{n-1}}(b^{+})^{m-1}R^{+}|0\rangle= \\
   \\
=\overbrace{(-1)^{n-1}(f^{+})^{n-1}Q^{+}}(b^{+})^{m-1}R^{+}|0\rangle=(-1)^{n-1}(f^{+})^{n-1}\underbrace{Q^{+}(b^{+})^{m-1}}R^{+}|0\rangle= \\
  \\
=(-1)^{n-1}(f^{+})^{n-1}\underbrace{\Big((b^{+})^{m-1}Q^{+}+\frac{(1-(-1)^{m-1})}{2}f^{+}(b^{+})^{m-2}+} \\
\underbrace{+2\big(\lfloor\frac{m-3}{2}\rfloor+1\big)(b^{+})^{m-3} R^{+}\Big)}R^{+}|0\rangle=  \\
  \\
\textrm{Using the nilpotency } (R^{+})^{2}=0 \textrm{ (which is a direct consequence of the algebraic} \\
\textrm{relations \equref{mixed}), together with } \{Q^{+},R^{+}\}=0 \textrm{ and } Q^{+}|0\rangle=0 \textrm{ the first and the third } \\
\textrm{summands of the above parenthesis are annihilated and only the middle-term survives:}  \\
  \\
=\frac{((-1)^{n-1}-(-1)^{n+m-2})}{2}(f^{+})^{n}(b^{+})^{m-2}R^{+}|0\rangle=  \\
  \\
=\frac{((-1)^{n-1}-(-1)^{n+m})}{2}|m-1,n+1,\beta\rangle
\end{array}
$$
}
for $m\geq2$ and $n\geq1$. Finally, we can directly verify that $Q^{+}|1,n,\beta\rangle=0$ completing thus the proof of the second one of \equref{7}. Recall that the $\psi_{m,n}\equiv|m,n,\beta\rangle$ vectors vanish for either $m=0$ or $n=0$.
\end{proof}
Let us now recall the relations:
\begin{equation}  \eqlabel{8}
\begin{array}{ccccc}
\big[ N_{b}, f^{+} \big] = 0  &  & , &  &  \big[ N_{f}, b^{+} \big] = 0
\end{array}
\end{equation}
which are clearly a rewriting of the corresponding relations from \equref{mixed} and \equref{pure}. Now we can get
\begin{equation}  \eqlabel{9}
\begin{array}{ccccc}
\big[ N_{b}, (b^{+})^{n} \big] = n(b^{+})^{n}  &  & , &  & \big[ N_{f}, (f^{+})^{n} \big] = n(f^{+})^{n}  \\
\end{array}
\end{equation}
Relations \equref{9} can be inductively proved from relations \equref{8}. Combining these last relations together with the relations $[N_{b}, R^{+}] = [N_{f}, R^{+}] = R^{+}$ we get the following
\begin{proposition}   \prlabel{NbNfaction}
We have the following relations regarding the action of the $N_{b}$, $N_{f}$ operators
\begin{equation}   \eqlabel{10}
\begin{array}{cccc}
\mathbf{N_{b}|m,n,\alpha\rangle=m|m,n,\alpha\rangle}, &  & & \mathbf{N_{b}|m,n,\beta\rangle=m|m,n,\beta\rangle}   \\
    \\
\mathbf{N_{f}|m,n,\alpha\rangle=n|m,n,\alpha\rangle}, &  & & \mathbf{N_{f}|m,n,\beta\rangle=n|m,n,\beta\rangle}
\end{array}
\end{equation}
for $0 \leq m$ and $0 \leq n \leq p$. In other words the vectors $|m, n, \alpha \rangle$, $|m, n, \beta \rangle$ described in \thref{FockspstructPBF11} for all values of $0 \leq m$ and $0 \leq n \leq p$ are eigenvectors of the $N_{b}$, $N_{f}$ operators.
\end{proposition}
\begin{proof}
We will only prove the fourth of the above relations. The rest of the relations can be treated similarly:
$$
\begin{array}{c}
N_{f}|m,n,\beta\rangle=\underbrace{N_{f}(f^{+})^{n-1}}(b^{+})^{m-1}R^{+}|0\rangle= \\
  \\
=\underbrace{\big((f^{+})^{n-1}N_{f}+(n-1)(f^{+})^{n-1}\big)}(b^{+})^{m-1}R^{+}|0\rangle=
\end{array}
$$
$$
\begin{array}{c}=(f^{+})^{n-1}N_{f}(b^{+})^{m-1}R^{+}|0\rangle+(n-1)(f^{+})^{n-1}(b^{+})^{m-1}R^{+}|0\rangle=  \\
   \\
=(f^{+})^{n-1}(b^{+})^{m-1}N_{f}R^{+}|0\rangle+(n-1)(f^{+})^{n-1}(b^{+})^{m-1}R^{+}|0\rangle=    \\
    \\
\textrm{applying } \big[N_{f},R^{+}\big]=R^{+} \textrm{ and } N_{f}|0\rangle=0 \textrm{ in the first summand} \\
   \\
=(f^{+})^{n-1}(b^{+})^{m-1}R^{+}|0\rangle+(n-1)(f^{+})^{n-1}(b^{+})^{m-1}R^{+}|0\rangle= \\
 \\
|m,n,\beta\rangle+(n-1)|m,n,\beta\rangle=n|m,n,\beta\rangle
\end{array}
$$
which completes the proof.
\end{proof}
Before closing this paragraph let us mention that all the computational results of the present section i.e. all the formulae of \prref{Q-action} describing the $Q^{-}$ action, the formulae of \prref{Q+action} describing the $Q^{+}$ action and the formulae of \prref{NbNfaction} describing the action of $N_{b}, N_{f}$ together with \equref{equat3} describing the action of the $P_{BF}$ generators, apart from having been proved analytically in the text, have also been checked and verified with the help of the \textsc{Quantum} \cite{GMFD} add-on for Mathematica 7.0: This is an add-on permitting symbolic algebraic computations between operators satisfying a given set of relations, including the definition and use of a kind of generalized Dirac notation for the vectors of the carrier space. What we have verified with the aid of the above package is that all the action formulae displayed in the text always preserve the corresponding algebraic relations. We should also mention that the formulae of \prref{Q-action}, \prref{Q+action} and \prref{NbNfaction} could have also been proved by a direct application of the formulae \equref{equat3} which describe explicitly the action of the generators of $P_{BF}^{(1,1)}$ on the carrier spaces $\bigoplus_{n=0}^{p} \bigoplus_{m=0}^{\infty}\mathcal{V}_{m,n}$ (for any positive integer $p$).

\paragraph{$\bullet$ \textbf{Construction of the Lie superalgebra representation:}}

In this paragraph, we shall consider a special case of the paraparticle realizations constructed in \prref{parapartrealiz}. We will consider the case in which the carrier (super) space $V = V_{0} \oplus V_{1}$ of the (super) representation $P: \mathbb{U}(L) \rightarrow \mathcal{E}nd_{gr}(V)$ is 2-dimensional.
So, let $L=L_{0} \oplus L_{1}$ be any complex Lie superalgebra of either finite or infinite dimension and let $V = V_{0} \oplus V_{1}$ be a 2-dimensional, complex, super-vector space i.e. $dim_{\mathbb{C}}V_{0} = dim_{\mathbb{C}}V_{1} = 1$. If $V$ is the carrier space for a super-representation (or: a $\mathbb{Z}_{2}$-graded representation) of $L$, this is equivalent to the existence of an homomorphism $P: \mathbb{U}(L) \rightarrow \mathcal{E}nd_{gr}(V)$ of associative superalgebras, from $\mathbb{U}(L)$ to the algebra $\mathcal{E}nd_{gr}(V)$ of $\mathbb{Z}_{2}$-graded linear maps on $V$. Thus, for any homogeneous element $z \in L$ the image $P(z)$ will be a  $2 \times 2$ matrix of the form
{\small
\begin{equation}    \eqlabel{2dsuperepr}
\begin{array}{ccc}
 P(z)= \left(\begin{array}{c|c}
              A(z) & B(z) \\ \hline
              C(z) & D(z)
             \end{array}\right)    &  \begin{array}{c}
                                      \nearrow  \\ \\
                                      \searrow
                                      \end{array}   & \begin{array}{c}
                                                      P(X)=
                                                            \left(\begin{array}{ccc}
                                                            A(X) & 0 \\
                                                            0 & D(X)
                                                            \end{array}\right)
                                                            \ \ (if \ z = X \rightsquigarrow \underline{even})
                                                            \\  \\  \\
                                                      P(Y) =
                                                            \left(\begin{array}{ccc}
                                                            0 & B(Y) \\
                                                            C(Y) & 0
                                                            \end{array}\right)
                                                           \ \ (if \ z = Y \rightsquigarrow \underline{odd})
                                                      \end{array}
\end{array}
\end{equation}
}
where $A(X)$, $B(Y)$, $C(Y)$, $D(X)$ are complex numbers (for any choice $X$ and $Y$ of even and odd respectively elements of $L$).

The relations \equref{JordgenPBFeqeven}, \equref{JordgenPBFeqodd} of \prref{parapartrealiz} will be written
\begin{equation} \eqlabel{JordgenPBFeqeven2d}
\begin{array}{c}
J_{P_{BF}}(X_{i}) = \frac{1}{2} A(X_{i}) \{ b^{+}, b^{-} \} + \frac{1}{2} D(X_{i}) [ f^{+}, f^{-} ] = \\ \\
\ \ \ \ \ \ =A(X_{i}) N_{b} + D(X_{i}) N_{f} + \big( A(X_{i}) - D(X_{i}) \big) \frac{p}{2}
\end{array}
\end{equation}
for any even element ($Z=X_{i}$) of an homogeneous basis of $L$ and by
{\small
\begin{equation} \eqlabel{JordgenPBFeqodd2d}
J_{P_{BF}}(Y_{j})=\frac{1}{2}\Big(B(Y_{j})\{ b^{+},f^{-}\}+C(Y_{j})\{f^{+},b^{-}\}\Big)=B(Y_{j}) Q^{-} + C(Y_{j}) Q^{+}
\end{equation}
}
for any odd element ($Z=Y_{j}$) of an homogeneous basis of $L$.

Now we have the following proposition
\begin{proposition}   \prlabel{finalformulaerepresL}
Any complex Lie superalgebra $\mathbb{L}$ (of either finite or infinite dimension) possesing a 2-dimensional super (i.e.: a $\mathbb{Z}_{2}$-graded) representation of the form \equref{2dsuperepr} has also a family of infinite dimensional representations with carrier spaces
$$
\bigoplus_{n=0}^{p} \bigoplus_{m=0}^{\infty} \mathcal{V}_{m,n}
$$
the (infinite) dimensional vector spaces described in \thref{FockspstructPBF11}, and actions given by:
{\small
\begin{equation} \eqlabel{Liesuperalgebrarepresentation}
\boxed{
\begin{array}{c}
\bullet \mathbf{X_{i} \triangleright | m, n, \alpha \rangle} = \Big( \big( m + \frac{p}{2} \big)A(X_{i}) + \big( n - \frac{p}{2} \big)D(X_{i})  \Big) \mathbf{| m, n, \alpha \rangle}    \\
   \\
\bullet \mathbf{X_{i} \triangleright | m, n, \beta \rangle} = \Big( \big( m + \frac{p}{2} \big)A(X_{i}) + \big( n - \frac{p}{2} \big)D(X_{i})  \Big) \mathbf{| m, n, \beta \rangle}   \\
   \\
\bullet \mathbf{Y_{j} \triangleright | m, n, \alpha \rangle} = B(Y_{j}) \Big( (-1)^{n-1} n \mathbf{| m+1, n-1, \alpha \rangle} + (-1)^{n}n(n-1)\mathbf{| m+1, n-1, \beta \rangle} \Big) +  \\
 + C(Y_{j}) \Big( \frac{\big( (-1)^{n} - (-1)^{n+m} \big)}{2} \mathbf{| m-1, n+1, \alpha \rangle} + 2(-1)^{n}\big( \lfloor \frac{m-2}{2} \rfloor + 1 \big) \mathbf{| m-1, n+1, \beta \rangle} \Big)   \\
   \\
\bullet \mathbf{Y_{j} \triangleright | m, n, \beta \rangle} = B(Y_{j}) \Big( (-1)^{n-1} \mathbf{| m+1, n-1, \alpha \rangle} + (-1)^{n} (n-1) \mathbf{| m+1, n-1, \beta \rangle} \Big) +   \\
+ C(Y_{j}) \Big( \frac{\big( (-1)^{n-1} - (-1)^{n+m-2} \big)}{2} \mathbf{| m-1, n+1, \beta \rangle} \Big)
\end{array}
}
\end{equation}
}
for $0 \leq m$ and $0 \leq n \leq p$. In other words the vectors $|m, n, \alpha \rangle$, $|m, n, \beta \rangle$ described in \thref{FockspstructPBF11} for all values of $0 \leq m$ and $0 \leq n \leq p$, are eigenvectors under the action of the even elements of the Lie superalgebra. (In the above $p$ is considered to be an arbitrary -but fixed- positive integer, parametrizing the family of representations).
\end{proposition}
\begin{proof}
If we combine relation \equref{JordgenPBFeqeven2d} with the relations \equref{10} of \prref{NbNfaction} and relation \equref{JordgenPBFeqodd2d} with the relations \equref{4} of \prref{Q-action} and relations \equref{7} of \prref{Q+action} we arrive at the displayed formulae \equref{Liesuperalgebrarepresentation}.
\end{proof}
At this point we would like to investigate the reducibility of the Lie superalgebra representations constructed in \prref{finalformulaerepresL}. We have the following
\begin{proposition}   \prlabel{reducibility}
For any arbitrary (but fixed) value of the positive integer $p$, the representation of an arbitrary complex Lie superalgebra $\mathbb{L}$ constructed in \prref{finalformulaerepresL} possesses invariant subspaces, in other words it is a reducible representation. Furthermore, it is a decomposable representation, i.e. a decomposable $\mathbb{L}$-module.
\end{proposition}
\begin{proof}
In the following figure all the subspaces constituting the carrier space $\bigoplus_{n=0}^{p} \bigoplus_{m=0}^{\infty} \mathcal{V}_{m,n}$ of the representation are shown
{\scriptsize
\begin{equation}  \eqlabel{invariantsubspaces0}
\left(
  \begin{array}{ccccccccccccc}
    \mathcal{V}_{0,0} & \mathcal{V}_{0,1} & \ldots & \ldots & \mathcal{V}_{0,n} & \ldots & \ldots & \boxed{\mathbf{\mathcal{V}_{0,s}}} & \ldots & \ldots & \mathcal{V}_{0,p-1} & \mathcal{V}_{0,p} \\
    \mathcal{V}_{1,0} & \mathcal{V}_{1,1} & \ldots & \ldots & \mathcal{V}_{1,n} & \ldots & \boxed{\mathbf{\mathcal{V}_{1,s-1}}} & \ldots & \ldots  & \ldots & \mathcal{V}_{1,p-1} & \mathcal{V}_{1,p} \\
    \vdots & \vdots & \vdots & \vdots & \vdots & \vdots & \vdots & \vdots & \vdots & \vdots & \vdots & \vdots  \\
    \ldots & \boxed{\mathbf{\mathcal{V}_{s-1,1}}} & \ldots & \ldots & \ldots & \ldots & \ldots & \ldots & \ldots & \ldots & \ldots & \ldots  \\
    \boxed{\mathbf{\mathcal{V}_{s,0}}} & \ldots & \ldots & \ldots & \ldots & \ldots & \ldots & \ldots & \ldots & \ldots & \ldots & \ldots   \\
    \vdots & \vdots & \vdots & \vdots & \vdots & \vdots & \vdots & \vdots & \vdots & \vdots & \vdots & \vdots \\
    \mathcal{V}_{k,0} & \ldots & \ldots & \ldots &  \ldots & \ldots & \ldots & \ldots & \ldots & \ldots & \ldots & \boxed{\mathbf{\mathcal{V}_{k,p}}} \\
    \mathcal{V}_{k+1,0} & \ldots & \ldots & \ldots & \ldots  & \ldots & \ldots & \ldots & \ldots & \ldots & \boxed{\mathbf{\mathcal{V}_{k+1,p-1}}} & \mathcal{V}_{k+1,p} \\
    \mathcal{V}_{k+2,0} & \ldots & \ldots & \ldots & \ldots  & \ldots & \ldots & \ldots & \ldots & \boxed{\mathbf{\mathcal{V}_{k+2,p-2}}} & \mathcal{V}_{k+2,p-1} & \mathcal{V}_{k+2,p} \\
    \vdots & \vdots & \vdots & \vdots & \vdots & \vdots & \vdots & \vdots & \vdots & \vdots & \vdots & \vdots  \\
    \mathcal{V}_{m,0} & \mathcal{V}_{m,1} & \ldots & \ldots  & \mathcal{V}_{m,n} & \ldots & \ldots & \ldots & \ldots & \ldots & \mathcal{V}_{m,p} & \vdots \\
    \vdots & \vdots & \vdots & \vdots & \vdots & \vdots & \vdots & \vdots & \vdots & \vdots & \vdots & \vdots  \\
    \vdots & \vdots & \vdots & \vdots & \vdots & \vdots & \vdots & \vdots & \vdots & \vdots & \vdots & \vdots  \\
    \boxed{\mathbf{\mathcal{V}_{k+p,0}}} & \ldots & \ldots & \ldots & \ldots & \ldots & \ldots & \ldots & \ldots & \ldots & \ldots & \mathcal{V}_{k+p,p} \\
    \vdots & \vdots & \vdots & \vdots & \vdots & \vdots & \vdots & \vdots & \vdots & \vdots & \vdots & \vdots
  \end{array}
\right)
\end{equation}
}

$\blacktriangleright$ The direct sum of \underline{the lower sequence} of the boxed subspaces, represented in \equref{invariantsubspaces0}, is shown in the following expression (where $k\geq0$)
\begin{equation}  \eqlabel{invariantsubspaces}
\bigoplus_{i=0}^{p} \mathcal{V}_{k+p-i, i}=\mathcal{V}_{k+p, 0}\oplus\mathcal{V}_{k+p-1, 1}\oplus\ldots\oplus\mathcal{V}_{k+1, p-1}\oplus\mathcal{V}_{k, p}
\end{equation}
where $dim(\bigoplus_{i=0}^{p}\mathcal{V}_{k+p-i, i})=2p$. The above subspace is finite dimensional and invariant under the $\mathbb{L}$-action \equref{Liesuperalgebrarepresentation}. Thus, the representation of the Lie superalgebra $\mathbb{L}$ is a reducible one.

$\blacktriangleright$ The direct sum of \underline{the upper sequence} of the boxed subspaces, represented in \equref{invariantsubspaces0}, is shown in the following expression ($0\leq s\leq p-1$)
\begin{equation} \eqlabel{invariantsubspaces2}
\bigoplus_{i=0}^{s}\mathcal{V}_{s-i,i}=\mathcal{V}_{0,s}\oplus\mathcal{V}_{1,s-1}\oplus\ldots\oplus\mathcal{V}_{s-1,1}\oplus\mathcal{V}_{s,0}
\end{equation}
where $dim(\bigoplus_{i=0}^{s}\mathcal{V}_{s-i,i})=2s$. The above subspace is another finite dimensional, invariant subspace under the $\mathbb{L}$-action \equref{Liesuperalgebrarepresentation}.

$\blacktriangleright$ Now we can see that the invariant subspaces defined above i.e. $\bigoplus_{i=0}^{p} \mathcal{V}_{k+p-i, i}$ given by \equref{invariantsubspaces} (for $k \geq 0$) and $\bigoplus_{i=0}^{s} \mathcal{V}_{s-i,i}$ given by \equref{invariantsubspaces2} (for $0 \leq s \leq p-1$), are mutually disjoint and moreover each one of them is disjoint with the sum of all the others. Furthermore their (direct) sum equals the whole carrier space:
\begin{equation}  \eqlabel{completereducibility}
\bigoplus_{n=0}^{p} \bigoplus_{m=0}^{\infty} \mathcal{V}_{m,n} = \bigoplus_{k=0}^{\infty} \left( \bigoplus_{i=0}^{p} \mathcal{V}_{k+p-i, i} \right) \oplus \bigoplus_{s=0}^{p-1} \left( \bigoplus_{i=0}^{s} \mathcal{V}_{s-i,i} \right)
\end{equation}
We have thus shown that the carrier space of the $\mathbb{L}$-representation \equref{Liesuperalgebrarepresentation} decomposes into a direct sum of finite dimensional, invariant subspaces which implies that the $\mathbb{L}$-representation defined by \equref{Liesuperalgebrarepresentation} is a reducible and more specifically a decomposable $\mathbb{L}$-module. Thus the proof is complete.
\end{proof}

Note that while the $\mathcal{V}_{m,n}$ subspaces, described in \thref{FockspstructPBF11}, are invariant under the $\mathbb{L}_{0}$-action of \equref{Liesuperalgebrarepresentation} (i.e. the action of the even elements of the Lie superalgebra) however they are not invariant under the $\mathbb{L}_{1}$-action of \equref{Liesuperalgebrarepresentation} (i.e. the action of the odd elements of the superalgebra). On the other hand, the subspaces $\bigoplus_{i=0}^{p} \mathcal{V}_{k+p-i, i}$ (for $k \geq 0$) and $\bigoplus_{i=0}^{s} \mathcal{V}_{s-i,i}$ (for $0 \leq s \leq p-1$), defined by \equref{invariantsubspaces} and \equref{invariantsubspaces2} respectively, are invariant under the action of both the even and the odd elements of the superalgebra \equref{Liesuperalgebrarepresentation}.

Before closing this paragraph, let us quickly say a few words about the grading: Since we are dealing -in general- with representations of a LS i.e. a $\mathbb{Z}_{2}$-graded Lie algebra, it would be nice if we could describe a suitable $\mathbb{Z}_{2}$-grading for the $\bigoplus_{n=0}^{p} \bigoplus_{m=0}^{\infty} \mathcal{V}_{m,n}$ carrier space. Here of course, \emph{suitable grading of the carrier space} should be interpreted as a grading which will be compatible with the LS grading in such a way that the above described representations become $\mathbb{Z}_{2}$-graded modules. It turns out (see also the examples of the next paragraph) that there are two possible alternatives which satisfy the above requirements. These are both represented in the following figures. We can either assign even-odd grading alternating on the rows of subspaces:
{\scriptsize
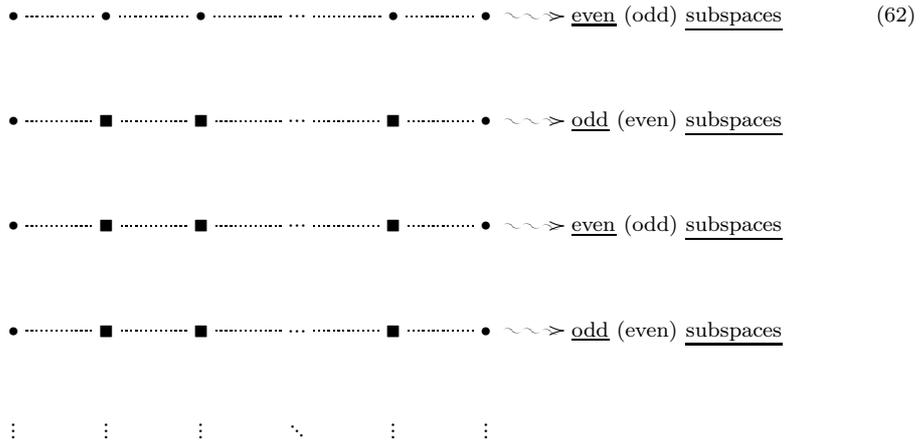
\begin{equation}  \eqlabel{grad1}
\xymatrix{
 \bullet\ar@{..}[r] & \bullet\ar@{..}[r]  & \bullet\ar@{..}[r] & \cdots\ar@{..}[r] & \bullet\ar@{..}[r] & \bullet\ar@{~>}[r] &  \textrm{\underline{even} (odd) \underline{subspaces}} \\
 \bullet\ar@{..}[r] & \blacksquare\ar@{..}[r]  & \blacksquare\ar@{..}[r] & \cdots\ar@{..}[r] & \blacksquare\ar@{..}[r] & \bullet\ar@{~>}[r] &  \textrm{\underline{odd} (even) \underline{subspaces}} \\
 \bullet\ar@{..}[r] & \blacksquare\ar@{..}[r]  & \blacksquare\ar@{..}[r] & \cdots\ar@{..}[r] & \blacksquare\ar@{..}[r] & \bullet\ar@{~>}[r] &  \textrm{\underline{even} (odd) \underline{subspaces}} \\
 \bullet\ar@{..}[r] & \blacksquare\ar@{..}[r]  & \blacksquare\ar@{..}[r] & \cdots\ar@{..}[r] & \blacksquare\ar@{..}[r] & \bullet\ar@{~>}[r] &  \textrm{\underline{odd} (even) \underline{subspaces}} \\
 \vdots & \vdots & \vdots & \ddots & \vdots & \vdots
}
\end{equation}
}
or we can assign even-odd grading by alternating on the columns of subspaces:
{\scriptsize
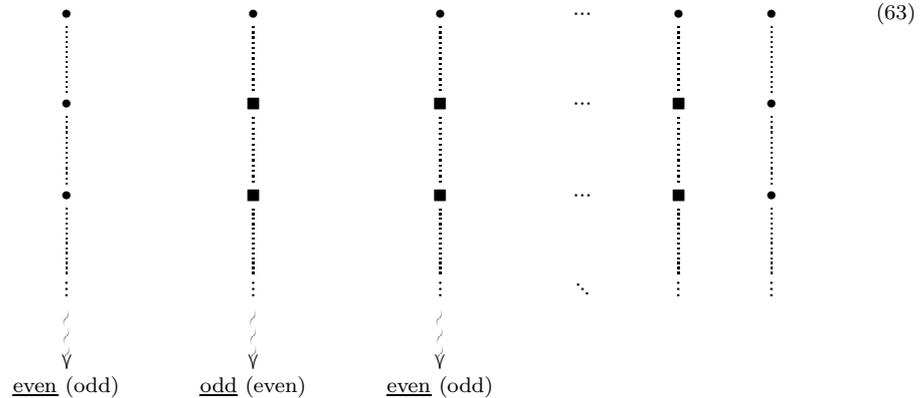
\begin{equation}  \eqlabel{grad2}
\xymatrix{
 \bullet\ar@{..}[d] & \bullet\ar@{..}[d]  & \bullet\ar@{..}[d] & \cdots & \bullet\ar@{..}[d] & \bullet\ar@{..}[d]  \\
 \bullet\ar@{..}[d] & \blacksquare\ar@{..}[d]  & \blacksquare\ar@{..}[d] & \cdots & \blacksquare\ar@{..}[d] & \bullet\ar@{..}[d]  \\
 \bullet\ar@{..}[d] & \blacksquare\ar@{..}[d]  & \blacksquare\ar@{..}[d] & \cdots & \blacksquare\ar@{..}[d] & \bullet\ar@{..}[d]   \\
 \vdots\ar@{~>}[d] & \vdots\ar@{~>}[d] & \vdots\ar@{~>}[d] & \ddots & \vdots & \vdots \\
 \textrm{\underline{even} (odd)}  & \textrm{\underline{odd} (even)} & \textrm{\underline{even} (odd)} &  &  &
}
\end{equation}
}
We have a number of remarks to make on the above diagrams: \\
First of all notice that we have denoted with a bullet $``\bullet"$ the $\mathcal{V}_{m,n}$ subspaces for which either $m=0$ or $n=0$ or $n=p$ i.e. the $1$-dimensional subspaces, while we have denoted by a black square $``\blacksquare"$ the $\mathcal{V}_{m,n}$ subspaces for which $m\neq0$ and $n\neq0$ and $n\neq p$ i.e. the $2$-dimensional subspaces.  \\
Second, we should emphasize that the values for the degrees of the subspaces assigned to the above diagrams define four inequivalent $\mathbb{Z}_{2}$-gradings for the carrier space $\bigoplus_{n=0}^{p} \bigoplus_{m=0}^{\infty} \mathcal{V}_{m,n}$. The first pair of gradings (represented in \equref{grad1}) alternate the degrees on the succeeding rows (i.e. define the degrees depending on the value of $m$) while the second pair (represented in \equref{grad2}) alternate the degrees on the succeeding columns (i.e. define the degrees depending on the value of $n$). The degree of the initial row or column is of course arbitrary. \\
Finally, in \equref{grad2} the degree of the last column of subspaces depends on whether the value of $p$ is even or odd.

\section{Applications: Some decompositions with respect to low-dimensional Lie algebras and superalgebras}   \selabel{applLSrepr}

In the last part of \seref{realizations} we have described the following chains of inclusions of subalgebras:
\begin{equation}  \eqlabel{iclusions}
\begin{array}{c}
\mathbb{U}(\mathfrak{gl}(m/n))\subset P_{BF} \\
  \\
\mathbb{U}(\mathfrak{sp}(2m))\subset \mathbb{U}(\mathfrak{osp}(1/2m))\subset P_{BF} \\
  \\
\mathbb{U}(\mathfrak{so}(2n))\subset \mathbb{U}(\mathfrak{so}(2n+1))\subset P_{BF} \\
    \\
\mathbb{U}\big(\mathfrak{so}(2n)\oplus \mathfrak{sp}(2m)\big)\cong \mathbb{U}(\mathfrak{so}(2n))\otimes \mathbb{U}(\mathfrak{sp}(2m))\subset \mathbb{U}(L_{00}\oplus L_{01})\subset P_{BF}
\end{array}
\end{equation}
Specifying for the $P_{BF}^{(1,1)}$ algebra, for which a class of irreducible $(\mathbb{Z}_{2}\times\mathbb{Z}_{2})$-graded representations, called Fock-like representations and parametrized by the positive integer $p$ has been described in \seref{singleFockspace}, the above chains of inclusions get their simplest form, i.e. for $m=n=1$ \equref{iclusions} becomes:
\begin{equation}  \eqlabel{iclusions2}
\begin{array}{c}
\mathbb{U}(\mathfrak{gl}(1/1))\subset P_{BF}^{(1,1)} \\
  \\
\mathbb{U}(\mathfrak{sp}(2))\subset \mathbb{U}(\mathfrak{osp}(1/2))\subset P_{BF}^{(1,1)} \\
  \\
\mathbb{U}(\mathfrak{so}(2))\subset \mathbb{U}(\mathfrak{so}(3))\subset P_{BF}^{(1,1)} \\
    \\
\mathbb{U}\big(\mathfrak{so}(2)\oplus \mathfrak{sp}(2)\big)\cong \mathbb{U}(\mathfrak{so}(2))\otimes \mathbb{U}(\mathfrak{sp}(2))\subset \mathbb{U}(L_{00}\oplus L_{01})\subset P_{BF}^{(1,1)}
\end{array}
\end{equation}
It is these \equref{iclusions2} according to which we are now going to present some decompositions of the initial Fock-like space $\bigoplus_{n=0}^{p} \bigoplus_{m=0}^{\infty} \mathcal{V}_{m,n}$. In the next diagrams, following the conventions of \seref{Focklikemodule}  we are going to use bullets $``\bullet"$ or $``\circ"$ to denote $1$-dimensional subspaces and squares $``\blacksquare"$ or $``\square"$ to denote $2$-dimensional subspaces. The filled or empty bullets or squares will be used to denote direct summands of different invariant subspaces i.e. the ``filled" and ``empty" symbols will denote direct summands of invariant subspaces which are not connected to each other through the action of the generators of the corresponding algebras. We will use the special symbol $``\filledstar"$ or $``\smallstar"$ to denote the $1$-dimensional subspace at the top left corner. This subspace is generated by the $|0\rangle$ vector which is usually referred to as the ``ground state" in the physics literature, thus $\filledstar\equiv\smallstar=<|0\rangle>$. Moreover, for the case of the Lie superalgebras ($\mathfrak{gl}(1/1)$, $L_{00}\oplus L_{01}$ and $\mathfrak{osp}(1/2)$) the action of the odd elements will be denoted by continuous arrows $\longrightarrow$ while the action of the even elements will be denoted by dashed arrows $\dashedrightarrow$ (notice that this is not the case for \equref{44} which is the usual Lie algebra $\mathfrak{sp}(2)$ thus all elements are even).

$\blacktriangleright$ \underline{Decompositions with respect to $\mathfrak{gl}(1/1)$:}  \\
The algebra is generated by
\begin{equation} \eqlabel{gen11}
\begin{array}{ccccc}
N_{b}, N_{f} \ \rightsquigarrow even  & & & & Q^{+}, Q^{-} \ \rightsquigarrow odd
\end{array}
\end{equation}
Its commutation-anticommutation relations in the above homogeneous basis read:
\begin{equation}
\begin{array}{c}
\big[N_{b},N_{f}\big]=\big[N_{b},N_{b}\big]=\big[N_{f},N_{f}\big]=0
\end{array}
\end{equation}
\begin{equation}
\begin{array}{ccccc}
\big[N_{b},Q^{\pm}\big]=\mp Q^{\pm} & & , & & \big[N_{f},Q^{\pm}\big]=\pm Q^{\pm}
\end{array}
\end{equation}
\begin{equation}
\begin{array}{ccccc}
\{Q^{+},Q^{+}\}=\{Q^{-},Q^{-}\}=0  & & , & & \{Q^{+},Q^{-}\}=N_{b}+N_{f}
\end{array}
\end{equation}
It can be considered either directly as a subalgebra of $P_{BF}$ or (equivalently) as being produced by the realizations \equref{JordgenPBFeqeven2d}, \equref{JordgenPBFeqodd2d} where we substitute the defining $2\times2$ matrix representation \cite{Scheu3} of $\mathfrak{gl}(1/1)$. The decomposition reads:
{\small
\begin{equation}  \eqlabel{11}
\xymatrix{
    \filledstar\ar@(ur,ul)[]|{Q^{\pm}}\ar@(ul,dl)@{..>}[]|{N_{f},N_{b}} & \ar@<1ex>[dl]|-{Q^{-}}\bullet\ar@(ul,dl)@{..>}[]|{N_{b}}\ar@(ur,dr)@{..>}[]|{N_{f}}\ar@(ur,ul)[]|{Q^{+}}  & \ldots & \bullet & \ar@<1ex>[dl]|-{Q^{-}}\bullet\ar@(ul,dl)@{..>}[]|{N_{b}}\ar@(ur,dr)@{..>}[]|{N_{f}}\ar@(ur,ul)[]|{Q^{+}} & \ldots  & \bullet & \bullet  \\
    \bullet\ar@<1ex>[ur]|-{Q^{+}}\ar@(ul,dl)[]|{Q^{-}}\ar@(dl,dr)@{..>}[]|{N_{b},N_{f}} & \blacksquare & \ldots & \ar@<1ex>[dl]|-{Q^{-}}\blacksquare\ar@<1ex>[ur]|-{Q^{+}}\ar@(ul,dl)@{..>}[]|{N_{b}}\ar@(ur,dr)@{..>}[]|{N_{f}} & \blacksquare & \ldots  & \blacksquare & \ar@<1ex>[dl]|-{Q^{-}}\bullet\ar@(ul,dl)@{..>}[]|{N_{b}}\ar@(ur,dr)@{..>}[]|{N_{f}}\ar@(ur,ul)[]|{Q^{+}} \\
    \vdots & \vdots & \ar@<1ex>[dl]|-{Q^{-}}\ddots\ar@<1ex>[ur]|-{Q^{+}} & \vdots & \vdots & \vdots & \ar@<1ex>[ur]|-{Q^{+}}\ddots\ar@<1ex>[dl]|-{Q^{-}} & \vdots  \\
    \bullet & \ar@<1ex>[dl]|-{Q^{-}}\blacksquare\ar@<1ex>[ur]|-{Q^{+}}\ar@(ul,dl)@{..>}[]|{N_{b}}\ar@(ur,dr)@{..>}[]|{N_{f}} & \ldots & \blacksquare & \blacksquare & \ar@<1ex>[dl]|-{Q^{-}}\cdots\ar@<1ex>[ur]|-{Q^{+}} & \blacksquare & \bullet  \\
    \bullet\ar@<1ex>[ur]|-{Q^{+}}\ar@(ul,dl)[]|{Q^{-}}\ar@(dl,dr)@{..>}[]|{N_{b},N_{f}} & \blacksquare & \ldots & \blacksquare & \ar@<1ex>[dl]|-{Q^{-}}\blacksquare\ar@<1ex>[ur]|-{Q^{+}}\ar@(ul,dl)@{..>}[]|{N_{b}}\ar@(ur,dr)@{..>}[]|{N_{f}} & \ldots & \blacksquare & \bullet   \\
    \vdots & \vdots & \ddots & \vdots\ar@<1ex>[ur]|-{Q^{+}} & \vdots & \ddots & \vdots & \vdots
    }
\end{equation}
}
Notice that the ``diagonal" invariant subspaces shown in \equref{11} are not irreducible. Their vector space dimension (over $\mathbb{C}$) ranges in the even values between $2$ and $2p$, except the $\filledstar$ space in the top left corner which is $1$-dimensional. They are $\mathbb{Z}_{2}$-graded representations  according to either \equref{grad1} or \equref{grad2}. Each one of them (except the $\filledstar$ invariant subspace) can further be decomposed into a direct sum of $2$-dimensional invariant subspaces. The exact branching rules together with multiplicities formulae will be given elsewhere.

$\blacktriangleright$ \underline{Decompositions with respect to the LS $L_{00}\oplus L_{01}$ whose even part is isomorphic}
\underline{to $\mathfrak{sp}(2)\oplus \mathfrak{so}(2)\cong \mathfrak{sp}(2)\oplus \mathfrak{gl}(1)$:}  \\
This LS has been described in \prref{subalgPbf} and \leref{commrel}. It is generated by:
\begin{equation} \eqlabel{gen22}
\begin{array}{ccccc}
N_{b}, N_{f}, (b^{+})^{2}, (b^{-})^{2} \ \rightsquigarrow even  & & & & Q^{+}, Q^{-}, R^{+}, R^{-} \ \rightsquigarrow odd
\end{array}
\end{equation}
Its comm.-anticom. relations, in the above homogeneous basis, and the decomposition read:
{\small
\begin{equation}
\begin{array}{c}

\begin{array}{c}
\big[N_{b},N_{f}\big]=\big[N_{b},N_{b}\big]=\big[N_{f},N_{f}\big]=0
\end{array}   \\
\\
\begin{array}{ccc}
\big[N_{b},(b^{\pm})^{2}\big]=\pm2(b^{\pm})^{2}, & \big[N_{f},(b^{\pm})^{2}\big]=0, &  \big[(b^{+})^{2},(b^{-})^{2}\big]=-4N_{b}-2p
\end{array}

\end{array}
\end{equation}
\begin{equation}
\begin{array}{c}

\begin{array}{cccc}
\big[N_{b},Q^{\pm}\big]=\mp Q^{\pm}, & \big[N_{f},Q^{\pm}\big]=\pm Q^{\pm}, & \big[N_{b},R^{\pm}\big]=\pm R^{\pm} & \big[N_{f},R^{\pm}\big]=\pm R^{\pm}
\end{array}\\
    \\
\begin{array}{cccc}
\big[(b^{+})^{2},Q^{+}\big]=-2R^{+},  & \big[(b^{+})^{2},Q^{-}\big]=0,  &  \big[(b^{+})^{2},R^{+}\big]=0,  & \big[(b^{+})^{2},R^{-}\big]=-2Q^{-}
\end{array}   \\
   \\
\begin{array}{cccc}
\big[(b^{-})^{2},Q^{+}\big]=0,  & \big[(b^{-})^{2},Q^{-}\big]=2R^{-},  & \big[(b^{-})^{2},R^{+}\big]=2Q^{+},  & \big[(b^{-})^{2},R^{-}\big]=0
\end{array}

\end{array}
\end{equation}
\begin{equation}
\begin{array}{c}

\begin{array}{ccc}
   \{Q^{+},Q^{+}\}=\{Q^{-},Q^{-}\}=0 & , & \{R^{+},R^{+}\}=\{R^{-},R^{-}\}=0
\end{array}  \\
    \\
\begin{array}{ccc}
\{Q^{+},Q^{-}\}=N_{b}+N_{f},  &  \{Q^{+},R^{+}\}=0, &  \{Q^{+},R^{-}\}=(b^{-})^{2}
\\
\end{array}  \\
   \\
\begin{array}{ccc}
\{Q^{-},R^{+}\}=(b^{+})^{2},  &  \{Q^{-},R^{-}\}=0, &  \{R^{+},R^{-}\}=N_{b}-N_{f}+p
\end{array}

\end{array}
\end{equation}
}
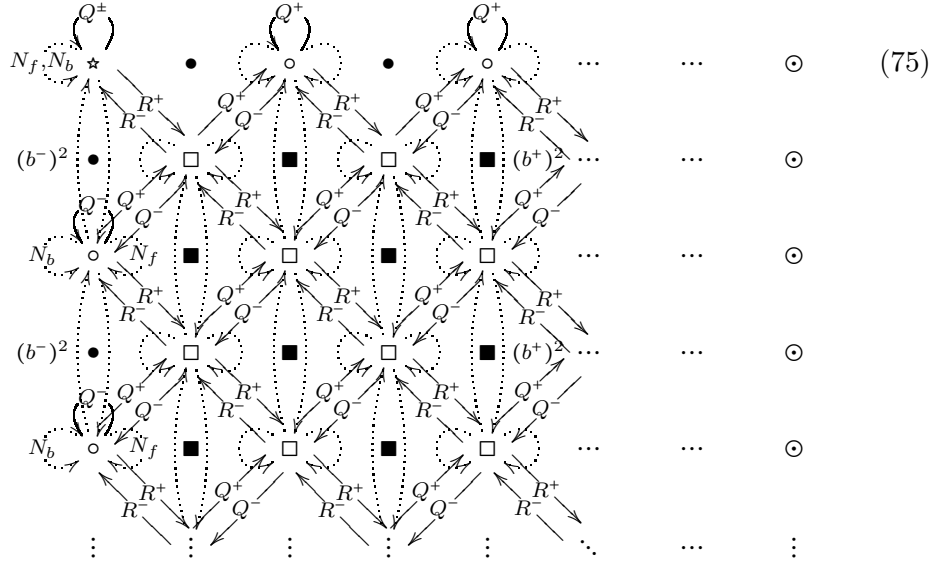
\begin{equation}   \eqlabel{22}
\xymatrix{
\smallstar\ar@(ur,ul)[]|{Q^{\pm}}\ar@(ul,dl)@{..>}[]|{N_{f},N_{b}}\ar@/^/@{..>}[dd]\ar@<1ex>[dr]|-{R^{+}}  & \bullet &  \circ\ar@(ur,dr)@{..>}[]\ar@(ul,dl)@{..>}[]\ar@/^/@{..>}[dd]\ar@<1ex>[dr]|-{R^{+}}\ar@<1ex>[dl]|-{Q^{-}}\ar@(ur,ul)[]|{Q^{+}} & \bullet & \circ\ar@(ur,dr)@{..>}[]\ar@(ul,dl)@{..>}[]\ar@/^/@{..>}[dd]^{(b^{+})^{2}}\ar@<1ex>[dr]|-{R^{+}}\ar@<1ex>[dl]|-{Q^{-}}\ar@(ur,ul)[]|{Q^{+}} & \cdots & \cdots & \odot \\
\bullet & \ar@<1ex>[ul]|-{R^{-}}\square\ar@(ur,dr)@{..>}[]\ar@(ul,dl)@{..>}[]\ar@<1ex>[dr]|-{R^{+}}\ar@/^/@{..>}[dd]\ar@<1ex>[ur]|-{Q^{+}}\ar@<1ex>[dl]|-{Q^{-}} & \blacksquare & \ar@<1ex>[ul]|-{R^{-}}\square\ar@(ur,dr)@{..>}[]\ar@(ul,dl)@{..>}[]\ar@/^/@{..>}[dd]\ar@<1ex>[dr]|-{R^{+}}\ar@<1ex>[ur]|-{Q^{+}}\ar@<1ex>[dl]|-{Q^{-}} & \blacksquare & \ar@<1ex>[ul]|-{R^{-}}\cdots\ar@<1ex>[dl]|-{Q^{-}} & \cdots & \odot \\
\circ\ar@(ur,dr)@{..>}[]|{N_{f}}\ar@(ul,dl)@{..>}[]|{N_{b}}\ar@/^/@{..>}[dd]\ar@/^/@{..>}[uu]^{(b^{-})^{2}}\ar@<1ex>[dr]|-{R^{+}}\ar@<1ex>[ur]|-{Q^{+}}
\ar@(ur,ul)[]|{Q^{-}} & \blacksquare & \ar@<1ex>[ul]|-{R^{-}}\square\ar@(ur,dr)@{..>}[]\ar@(ul,dl)@{..>}[]\ar@<1ex>[dr]|-{R^{+}}\ar@/^/@{..>}[dd]\ar@/^/@{..>}[uu]\ar@<1ex>[ur]|-{Q^{+}}\ar@<1ex>[dl]|-{Q^{-}} & \blacksquare & \ar@<1ex>[ul]|-{R^{-}}\square\ar@(ur,dr)@{..>}[]\ar@(ul,dl)@{..>}[]\ar@/^/@{..>}[dd]^{(b^{+})^{2}}\ar@/^/@{..>}[uu]\ar@<1ex>[dr]|-{R^{+}}\ar@<1ex>[ur]|-{Q^{+}}\ar@<1ex>[dl]|-{Q^{-}} & \cdots & \cdots & \odot \\
\bullet & \ar@/^/@{..>}[uu]\square\ar@(ur,dr)@{..>}[]\ar@(ul,dl)@{..>}[]\ar@/^/@{..>}[dd]\ar@<1ex>[ul]|-{R^{-}}\ar@<1ex>[dr]|-{R^{+}}\ar@<1ex>[ur]|-{Q^{+}}\ar@<1ex>[dl]|-{Q^{-}} & \blacksquare & \ar@<1ex>[ul]|-{R^{-}}\square\ar@(ur,dr)@{..>}[]\ar@(ul,dl)@{..>}[]\ar@<1ex>[dr]|-{R^{+}}\ar@/^/@{..>}[dd]\ar@/^/@{..>}[uu]\ar@<1ex>[ur]|-{Q^{+}}\ar@<1ex>[dl]|-{Q^{-}} & \blacksquare & \ar@<1ex>[ul]|-{R^{-}}\cdots\ar@<1ex>[dl]|-{Q^{-}} & \cdots & \odot \\
\circ\ar@(ur,dr)@{..>}[]|{N_{f}}\ar@(ul,dl)@{..>}[]|{N_{b}}\ar@/^/@{..>}[uu]^{(b^{-})^{2}}\ar@<1ex>[dr]|-{R^{+}}\ar@<1ex>[ur]|-{Q^{+}}\ar@(ur,ul)[]|{Q^{-}} & \blacksquare & \ar@<1ex>[ul]|-{R^{-}}\square\ar@(ur,dr)@{..>}[]\ar@(ul,dl)@{..>}[]\ar@/^/@{..>}[uu]\ar@<1ex>[dr]|-{R^{+}}\ar@<1ex>[ur]|-{Q^{+}}\ar@<1ex>[dl]|-{Q^{-}} & \blacksquare & \ar@<1ex>[ul]|-{R^{-}}\square\ar@(ur,dr)@{..>}[]\ar@(ul,dl)@{..>}[]\ar@<1ex>[dr]|-{R^{+}}\ar@/^/@{..>}[uu]\ar@<1ex>[ur]|-{Q^{+}}\ar@<1ex>[dl]|-{Q^{-}}  & \cdots & \cdots & \odot \\
\vdots & \ar@<1ex>[ul]|-{R^{-}}\vdots\ar@/^/@{..>}[uu]\ar@<1ex>[ur]|-{Q^{+}} & \vdots & \ar@<1ex>[ul]|-{R^{-}}\vdots\ar@/^/@{..>}[uu]\ar@<1ex>[ur]|-{Q^{+}} & \vdots & \ar@<1ex>[ul]|-{R^{-}}\ddots & \cdots & \vdots
}
\end{equation}
The $\odot$ symbol in the last column stands for either $\bullet$ or $\circ$ depending on the value of $p$ being odd or even integer respectively. The representations are $\mathbb{Z}_{2}$-graded representations according to the gradings described in either \equref{grad1} or \equref{grad2}. From the above figure we can readily see that the whole carrier space has a decomposition
\begin{equation}  \eqlabel{decfilempt}
\bigoplus_{n=0}^{p} \bigoplus_{m=0}^{\infty} \mathcal{V}_{m,n}=\oplus(\bullet\oplus\blacksquare) \bigoplus \oplus(\circ\oplus\square) \bigoplus \smallstar
\end{equation}
where $\oplus(\bullet\oplus\blacksquare)$ and $\oplus(\circ\oplus\square)\bigoplus\smallstar$ are the ``filled" and ``empty" invariant subspaces respectively. These are infinite dimensional submodules. It will be interesting to further investigate the decompositions of the above ``filled" or ``empty" infinite dimensional submodules and to compute the corresponding branching rules and multiplicities formulae.

$\blacktriangleright$ \underline{Decompositions with respect to $\mathfrak{osp}(1/2)$:}  \\
This is the well-known parabosonic algebra in a single degree of freedom with generators and commutation-anticommutation relations:
\begin{equation} \eqlabel{gen33}
\begin{array}{ccccc}
N_{b}, (b^{+})^{2}, (b^{-})^{2} \rightsquigarrow even  & & & & b^{+}, b^{-} \ \rightsquigarrow odd \\
\end{array}
\end{equation}
{\small
\begin{equation}
\begin{array}{cccc}
 \big[N_{b},(b^{\pm})^{2}\big]=\pm2(b^{\pm})^{2}, & \big[(b^{+})^{2},(b^{-})^{2}\big]=-4N_{b}-2p & \{b^{+},b^{-}\}=2N_{b}+p, & \big[N_{b},b^{\pm}\big]=\pm b^{\pm}
\end{array}
\end{equation}
}
Finally, the decomposition of the $\bigoplus_{n=0}^{p} \bigoplus_{m=0}^{\infty} \mathcal{V}_{m,n}$ carrier space becomes:
\begin{equation}  \eqlabel{33}
\xymatrix{
\filledstar\ar@(ul,dl)@{..>}[]|{N_{b}}\ar@/^/[d]|-{b^{+}} & \bullet\ar@(ul,dl)@{..>}[]|{N_{b}}\ar@/^/[d]|-{b^{+}}  & \bullet\ar@(ul,dl)@{..>}[]|{N_{b}}\ar@/^/[d]|-{b^{+}}  & \cdots & \cdots  & \bullet\ar@(ul,dl)@{..>}[]|{N_{b}}\ar@/^/[d]|-{b^{+}}  & \bullet\ar@(ul,dl)@{..>}[]|{N_{b}}\ar@/^/[d]|-{b^{+}}   \\
\bullet\ar@(ul,dl)@{..>}[]|{N_{b}}\ar@/^/[d]|-{b^{+}}\ar@/^/[u]|-{b^{-}}  & \blacksquare\ar@(ul,dl)@{..>}[]|{N_{b}}\ar@/^/[d]|-{b^{+}}\ar@/^/[u]|-{b^{-}}  & \blacksquare\ar@(ul,dl)@{..>}[]|{N_{b}}\ar@/^/[d]|-{b^{+}}\ar@/^/[u]|-{b^{-}} & \cdots & \cdots  & \blacksquare\ar@(ul,dl)@{..>}[]|{N_{b}}\ar@/^/[d]|-{b^{+}}\ar@/^/[u]|-{b^{-}} & \bullet\ar@(ul,dl)@{..>}[]|{N_{b}}\ar@/^/[d]|-{b^{+}}\ar@/^/[u]|-{b^{-}}    \\
\bullet\ar@(ul,dl)@{..>}[]|{N_{b}}\ar@/^/[d]|-{b^{+}}\ar@/^/[u]|-{b^{-}}  & \blacksquare\ar@(ul,dl)@{..>}[]|{N_{b}}\ar@/^/[d]|-{b^{+}}\ar@/^/[u]|-{b^{-}}   & \blacksquare\ar@(ul,dl)@{..>}[]|{N_{b}}\ar@/^/[d]|-{b^{+}}\ar@/^/[u]|-{b^{-}}   & \cdots & \cdots  & \blacksquare\ar@(ul,dl)@{..>}[]|{N_{b}}\ar@/^/[d]|-{b^{+}}\ar@/^/[u]|-{b^{-}}   & \bullet\ar@(ul,dl)@{..>}[]|{N_{b}}\ar@/^/[d]|-{b^{+}}\ar@/^/[u]|-{b^{-}}   \\
\vdots\ar@/^/[d]|-{b^{+}}\ar@/^/[u]|-{b^{-}}  & \vdots\ar@/^/[d]|-{b^{+}}\ar@/^/[u]|-{b^{-}}   & \vdots\ar@/^/[d]|-{b^{+}}\ar@/^/[u]|-{b^{-}}   & \cdots & \cdots  & \vdots\ar@/^/[d]|-{b^{+}}\ar@/^/[u]|-{b^{-}}   & \vdots\ar@/^/[d]|-{b^{+}}\ar@/^/[u]|-{b^{-}}   \\
\bullet\ar@(ul,dl)@{..>}[]|{N_{b}}\ar@/^/[d]|-{b^{+}}\ar@/^/[u]|-{b^{-}}  & \blacksquare\ar@(ul,dl)@{..>}[]|{N_{b}}\ar@/^/[d]|-{b^{+}}\ar@/^/[u]|-{b^{-}}  & \blacksquare\ar@(ul,dl)@{..>}[]|{N_{b}}\ar@/^/[d]|-{b^{+}}\ar@/^/[u]|-{b^{-}}  & \cdots & \cdots  & \blacksquare\ar@(ul,dl)@{..>}[]|{N_{b}}\ar@/^/[d]|-{b^{+}}\ar@/^/[u]|-{b^{-}}  & \bullet\ar@(ul,dl)@{..>}[]|{N_{b}}\ar@/^/[d]^{b^{+}}\ar@/^/[u]|-{b^{-}}   \\
\vdots\ar@/^/[u]|-{b^{-}} & \vdots\ar@/^/[u]|-{b^{-}}  & \vdots\ar@/^/[u]|-{b^{-}}  & \ddots & \ddots & \vdots\ar@/^/[u]|-{b^{-}}  & \vdots\ar@/^/[u]|-{b^{-}}
}
\end{equation}
The first and the last columns are isomorphic and correspond to the well-known \cite{LiStJ} Fock representation of the single degree of freedom parabosonic algebra, for an arbitrary value of the positive integer $p$. It is a $\mathbb{Z}_{2}$-graded, infinite dimensional, irreducible representation according to the grading \equref{grad1}, which alternates the degrees on the rows of the initial carrier space $\bigoplus_{n=0}^{p} \bigoplus_{m=0}^{\infty} \mathcal{V}_{m,n}$. For $p=1$ it stems from the irreducible Fock representation of the first Weyl algebra commonly known in the physics literature as the Canonical Commutation Relations algebra (CCR). Each of the above invariant columns (except the first and the last one) can be further decomposed in a sum of two irreducible Fock representations of the parabosonic algebra $P_{B}$, but with each one of these representations corresponding to a different value of $p$. From the physics viewpoint this implies that inside the mixed paraparticle system $P_{BF}$ ``live" free parabosons of different order! This is a result which will be discussed in detail in a forthcoming work.

$\blacktriangleright$ \underline{Decompositions with respect to $\mathfrak{sp}(2)$:}  \\
This is a subalgebra of the parabosonic algebra in a single degree of freedom and is generated by:
\begin{equation} \eqlabel{gen44}
\begin{array}{c}
N_{b}, (b^{+})^2, (b^{-})^2
\end{array}
\end{equation}
and commutation relations (in the above basis):
\begin{equation}
\begin{array}{ccc}
\big[N_{b},(b^{+})^{2}\big]=2(b^{+})^{2}, & \big[N_{b},(b^{-})^{2}\big]=-2(b^{-})^{2}, &  \big[(b^{+})^{2},(b^{-})^{2}\big]=-4N_{b}-2p
\end{array}
\end{equation}
The carrier space decomposes now according to the following figure
\begin{equation} \eqlabel{44}
\xymatrix{
\smallstar\ar@(ur,dr)[]|{N_{b}}\ar@/^/[dd] & \circ\ar@(ur,dr)[]|{N_{b}}\ar@/^/[dd] & \circ\ar@(ur,dr)[]|{N_{b}}\ar@/^/[dd] & \cdots &  \cdots & \circ\ar@(ur,dr)[]|{N_{b}}\ar@/^/[dd] & \circ\ar@(ur,dr)[]|{N_{b}}\ar@/^/[dd]^{(b^{+})^{2}}   \\
\bullet\ar@(ur,dr)@{..>}[]\ar@/^/@{..>}[dd] & \blacksquare\ar@(ur,dr)@{..>}[]\ar@/^/@{..>}[dd] & \blacksquare\ar@(ur,dr)@{..>}[]\ar@/^/@{..>}[dd]  & \cdots & \cdots & \blacksquare\ar@(ur,dr)@{..>}[]\ar@/^/@{..>}[dd] & \bullet\ar@(ul,ur)@{..>}[]\ar@/^/@{..>}[dd]   \\
\circ\ar@(ur,dr)[]|{N_{b}}\ar@/^/[dd]\ar@/^/[uu]^{(b^{-})^{2}} & \square\ar@(ur,dr)[]|{N_{b}}\ar@/^/[dd]\ar@/^/[uu] & \square\ar@(ur,dr)[]|{N_{b}}\ar@/^/[dd]\ar@/^/[uu] & \cdots & \cdots & \square\ar@(ur,dr)[]|{N_{b}}\ar@/^/[dd]\ar@/^/[uu] & \circ\ar@(ur,dr)[]|{N_{b}}\ar@/^/[dd]^{(b^{+})^{2}}\ar@/^/[uu]  \\
\bullet\ar@(ur,dr)@{..>}[]\ar@/^/@{..>}[dd]\ar@/^/@{..>}[uu] & \blacksquare\ar@(ur,dr)@{..>}[]\ar@/^/@{..>}[dd]\ar@/^/@{..>}[uu] & \blacksquare\ar@(ur,dr)@{..>}[]\ar@/^/@{..>}[dd]\ar@/^/@{..>}[uu] & \cdots & \cdots & \blacksquare\ar@(ur,dr)@{..>}[]\ar@/^/@{..>}[dd]\ar@/^/@{..>}[uu] & \bullet\ar@(ul,ur)@{..>}[]\ar@/^/@{..>}[dd]\ar@/^/@{..>}[uu]  \\
\circ\ar@(ur,dr)[]|{N_{b}}\ar@/^/[dd]\ar@/^/[uu]^{(b^{-})^{2}} & \square\ar@(ur,dr)[]|{N_{b}}\ar@/^/[dd]\ar@/^/[uu] & \square\ar@(ur,dr)[]|{N_{b}}\ar@/^/[dd]\ar@/^/[uu] & \cdots & \cdots & \square\ar@(ur,dr)[]|{N_{b}}\ar@/^/[dd]\ar@/^/[uu] & \circ\ar@(ur,dr)[]|{N_{b}}\ar@/^/[dd]^{(b^{+})^{2}}\ar@/^/[uu] \\
\vdots\ar@/^/@{..>}[uu] & \vdots\ar@/^/@{..>}[uu] & \vdots\ar@/^/@{..>}[uu] & \ddots & \ddots & \vdots\ar@/^/@{..>}[uu] & \vdots\ar@/^/@{..>}[uu]   \\
\vdots\ar@/^/[uu]^{(b^{-})^{2}} & \vdots\ar@/^/[uu] & \vdots\ar@/^/[uu]  & \ddots & \ddots & \vdots\ar@/^/[uu] & \vdots\ar@/^/[uu]
}
\end{equation}
Notice that we use continuous and dashed arrows to indicate the connection between direct summands of the ``isolated" invariant subspaces. The decomposition is being generated by branching the corresponding representations of $\mathfrak{osp}(1/2)$ given in \equref{33}. The first and the last columns contain -as can be clearly inferred from the figure above- two copies (each) of an infinite dimensional irreducible representation of $\mathfrak{sp}(2)$. Each of the rest columns decomposes as $(\circ\oplus\square)\bigoplus(\oplus\blacksquare)$ where each one of $\circ\oplus\square$ and $\oplus\blacksquare$ further contains two direct summands which are infinite dimensional, irreducible representations. We will come back with more details, branching rules and multiplicities in a forthcoming work.

$\blacktriangleright$ \underline{Decompositions with respect to $\mathfrak{so}(3)\cong \mathfrak{su}(2)$ and $\mathfrak{so}(2)\cong \mathfrak{gl}(1)$:} \footnote{Note that $\mathfrak{so}(3)\cong \mathfrak{su}(2)$ and $\mathfrak{so}(2)\cong \mathfrak{gl}(1)\cong \mathfrak{u}(1)$ as Lie algebras over $\mathbb{C}$}  \\
$\mathfrak{so}(3)\cong \mathfrak{su}(2)$ is the well-known parafermionic algebra in a single degree of freedom. Its generators are:
\begin{equation} \eqlabel{gen55}
\begin{array}{c}
N_{f}, f^{+}, f^{-}
\end{array}
\end{equation}
and its defining commutation relations read
\begin{equation}
\begin{array}{ccc}
 N_{f}=\frac{1}{2}\big[f^{+},f^{-}\big]+\frac{p}{2}, & \big[N_{f},f^{+}\big]=f^{+} & \big[N_{f},f^{-}\big]=-f^{-}
\end{array}
\end{equation}
while $N_{f}$ alone generates $\mathfrak{so}(2)\cong \mathfrak{gl}(1)$. We now get decompositions in rows:
\begin{equation}    \eqlabel{55}
\begin{array}{cccccccc}
\xymatrix{
 \filledstar\ar@/^/[r]|-{f^{+}}\ar@(ur,ul)[]|{N_{f}} & \bullet\ar@(ur,ul)[]|{N_{f}}\ar@/^/[r]|-{f^{+}}\ar@/^/[l]|-{f^{-}}  & \cdots\ar@/^/[l]|-{f^{-}}\ar@/^/[r]|-{f^{+}}  & \bullet\ar@(ur,ul)[]|{N_{f}}\ar@/^/[r]|-{f^{+}}\ar@/^/[l]|-{f^{-}}  &  \bullet\ar@(ur,ul)[]|{N_{f}}\ar@/^/[l]|-{f^{-}} \\
 \bullet\ar@/^/[r]\ar@(ur,ul)[] & \blacksquare\ar@(ur,ul)[]\ar@/^/[r]\ar@/^/[l]  & \cdots\ar@/^/[l]\ar@/^/[r]  & \blacksquare\ar@(ur,ul)[]\ar@/^/[r]\ar@/^/[l]  &  \bullet\ar@(ur,ul)[]\ar@/^/[l] \\
 \bullet\ar@/^/[r]\ar@(ur,ul)[] & \blacksquare\ar@(ur,ul)[]\ar@/^/[r]\ar@/^/[l]  & \cdots\ar@/^/[l]\ar@/^/[r]  & \blacksquare\ar@(ur,ul)[]\ar@/^/[r]\ar@/^/[l]  &  \bullet\ar@(ur,ul)[]\ar@/^/[l] \\
 \vdots & \vdots  & \ddots  & \vdots  & \vdots
}
 & & & & & &
\xymatrix{
 \filledstar\ar@(ur,ul)[]|{N_{f}} & \bullet\ar@(ur,ul)[]|{N_{f}}  & \cdots  &  \bullet\ar@(ur,ul)[]|{N_{f}} \\
 \bullet\ar@(ur,ul)[]|{N_{f}} & \blacksquare\ar@(ur,ul)[]|{N_{f}}  & \cdots  & \bullet\ar@(ur,ul)[]|{N_{f}}  \\
 \vdots & \vdots  & \ddots  & \vdots
}
\end{array}
\end{equation}
For the first of the above two diagrams notice that the first row corresponds to the well-known \cite{StJ} Fock representation of the single degree of freedom parafermionic algebra, for an arbitrary value of the positive integer $p$. It is a finite dimensional, irreducible representation. Its vector space dimension (over $\mathbb{C}$) is $p+1$. For $p=1$ it stems from the irreducible Fock representation the Canonical Anticommutation Relations algebra (CAR). Each of the above invariant rows (except the first one) can be further decomposed in a sum of two irreducible copies of the first row but with each of these copies corresponding to a different value of $p$. From the physics viewpoint this implies that inside the mixed paraparticle system $P_{BF}$ ``live" free parafermions of different order! This is a result -together with the corresponding remark for the $osp(1/2)$ decompositions- which deserve to be discussed in detail in a forthcoming work.

\section{Discussion}

In \seref{realizations} we have studied various algebraic properties of the Relative Parabose Set algebra among others: its structure as the UEA of a $\theta$-colored, $(\mathbb{Z}_{2}\times\mathbb{Z}_{2})$-graded Lie algebra, various Lie algebras and superalgebras contained as subalgebras of $P_{BF}$ and the construction of realizations of Lie algebras and superalgebras either via homomorphisms of the form $J_{P_{BF}}:\mathbb{U}(L)\rightarrow\mathbb{U}\big(\mathfrak{gl}(m/n)\big) \subset P_{BF}$ or via inclusions as subalgebras.

In \seref{singleFockspace}, we have reviewed a family of infinite dimensional, irreducible, $(\mathbb{Z}_{2}\times\mathbb{Z}_{2})$-graded Fock-like representations of $P_{BF}^{(1,1)}$, parameterized by a positive integer $p$. These have been first introduced in \cite{Ya2} and further developed and studied in \cite{KaDaHa3,KaDaHa4}.

In \seref{Focklikemodule} we have ``translated" the above mentioned irreducible paraparticle representations to representations of a Lie superalgebra possessing a $\mathbb{Z}_{2}$-graded, $2\times2$ matrix representation i.e. through homomorphisms of the form $J_{P_{BF}}:\mathbb{U}(L)\rightarrow\mathbb{U}\big(\mathfrak{gl}(1/1)\big)$. We have shown that such representations are in general decomposable and we have provided four schemes for consistent $\mathbb{Z}_{2}$-gradings.

In \seref{applLSrepr}, as an application of the above ideas and techniques, we have constructed the decompositions of the Fock-like spaces under the actions of various low-dimensional Lie algebras and superalgebras: we have studied the cases of $\mathfrak{gl}(1/1)$, $L_{00}\oplus L_{01}$ introduced in \seref{realizations}, $\mathfrak{osp}(1/2)$, $\mathfrak{sp}(2)$, $\mathfrak{so}(2)$ and $\mathfrak{so}(3)$. It is interesting to proceed in computing the exact branching rules and multiplicities formulae for the cases which have been studied and for other Lie (super)algebras as well. This will be done in the frame of a forthcoming work connecting these algebras and their representations with various ideas for applications in physics. The interested reader at this aspect should have a look at the ideas discussed in the last section of \cite{KaDaHa3}.

All the computational results of the present paper, apart from having been proved in the text, have also been checked and verified with the help of the \textsc{Quantum} \cite{GMFD} add-on for Mathematica 7.0: This is an add-on permitting symbolic algebraic computations between operators satisfying a given set of relations, including the definition and use of a kind of generalized Dirac notation for the vectors of the carrier space. What we have verified with the aid of the above package is that all the action formulae displayed in the text always preserve the corresponding algebraic relations.

Before closing this work we should mention a leading idea behind these techniques: From the results of \seref{realizations} it is evident that it will be really interesting to study the general case of the $P_{BF}$ algebra in infinite degrees of freedom and to attempt to make the connection with decompositions and representations of the general case of $\mathfrak{gl}(n/m)$, $\mathfrak{osp}(2n/2m)$, $\mathfrak{osp}(1/2n)$, $\mathfrak{sp}(2n)$, $\mathfrak{so}(2n+1)$ etc. Such a study, will provide us with deeper inside into the representation theory of various Lie algebras and superalgebras and will provide immediate connections with various diversified analytical tools of the physics literature. However, the study and the explicit construction of the irreducible Fock-like represenations of $P_{BF}$ is a highly non-trivial task: we just mention that the -much simpler- cases of the $P_{B}$ and $P_{F}$ algebras have been solved \cite{LiStJ,StJ} only recently. We intend to proceed in facing the above problem either using techniques of induced representations following the ideas of \cite{LiStJ,StJ} or to follow the ``braided" methodology proposed in \cite{KaDa3}.

\section*{Acknowledgements}

{\small
Part of this work was done while KK was a postdoctoral fellow researcher at the Institute of Physics and Mathematics (IFM) of the University of Michoacan (UMSNH) at Morelia, Michoacan, Mexico. The rest part was done during 2011 while KK was a postdoctoral fellow researcher at the School of Mathematics of the Aristotle University of Thessaloniki (AUTh). He has been supported by \textsc{Conacyt}/No. J60060 and by a scholarship from the Research Committee of the Aristotle University of Thessaloniki respectively. The research of AHA was supported by grants \textsc{Cic} 4.16 and \textsc{Conacyt}/No. J60060; he is also grateful to \textsc{Sni}. KK would like to thank Prof. Costas Daskaloyannis for reading through the manuscript and providing valuable observation and criticism. Both authors are grateful to Prof. Mustapha Lahyane for his interest in this work.
}


\end{document}